\documentclass[reqno,11pt,centertags]{article}
\usepackage{amsmath,amsthm,amscd,amssymb,latexsym,upref}
\input epsf
\usepackage{epsfig}
\usepackage[T2A,OT1]{fontenc}
\usepackage[ot2enc]{inputenc}
\usepackage[russian,english]{babel}
\usepackage{epic,eepicemu}

\newcommand{\bbR}{{\mathbb{R}}}

\newcommand{\bbC}{{\mathbb{C}}}

\newcommand{\cP}{{\mathcal{P}}}

\newcommand{\cM}{{\mathcal{M}}}

\def\bC{\overline{\mathbb{C}}}

\renewcommand{\Re}{\text{\rm Re\ }}
\renewcommand{\Im}{\text{\rm Im\ }}

\newcommand{\sgn}{\text{\rm sgn}}

\allowdisplaybreaks \numberwithin{equation}{section}
\newtheorem{theorem}{Theorem}[section]

\newtheorem{lemma}[theorem]{Lemma}
\newtheorem{proposition}[theorem]{Proposition}
\newtheorem{corollary}[theorem]{Corollary}

\theoremstyle{definition}

\date{\today}
\title
{Polynomials of the best uniform
approximation to $\sgn(x)$  on two intervals}
\author{Alexandre Eremenko\thanks{Supported by NSF grant
DMS-0555279.}$\;$
and Peter
Yuditskii\thanks{Supported by the Austrian Science Fund FWF,
project no: P22025--N18.}}
\begin{document}
\maketitle

\begin{abstract}
We describe polynomials of the best uniform approximation
to $\sgn(x)$ on the union of two intervals $[-A,-1]\cup[1,B]$ in
terms of special conformal mappings. This permits us to
find the exact asymptotic behavior of the error of this
approximation.

Uniform approximation, conformal mapping.
\end{abstract}

\section{Introduction}
In \cite{EY} we obtained precise asymptotics of
the error of the best polynomial approximation of
$\sgn(x)$ on two symmetric intervals $[-A,-1]\cup[1,A]$.
Paper \cite{NPVY} contains a somewhat simplified proof,
together with generalizations.
In this paper, we generalize the result to the case
of two arbitrary intervals, the problem proposed to us
by W.~Hayman and H.~Stahl, whom we thank.

Related problems on the asymptotics of the
error of the best uniform approximation
by polynomials of degree at most $n$ of the functions
$x^{n+1}$ and $1/(x-c),\; c\notin I$ on the union $I$ of two intervals
were completely
solved by N. I. Akhiezer in \cite{A1}.

Fuchs \cite{F1,F2,F3} studied general problems
of uniform polynomial approximation of piecewise
analytic functions on finite systems of intervals.
For the case of $\sgn(x)$
on two intervals $I=[-A,-1]\cup[1,B]$,
the result in \cite{F1} is
$$C_1n^{-1/2}e^{-\eta n}\leq L_n\leq C_2n^{-1/2}e^{-\eta n}.
$$
Here $$L_n=\inf_{p\in \cP_n}\sup_{x\in I}|\sgn(x)-p(x)|,$$
where $\cP_n$ is the set of polynomials of degree at most $n$;
positive constants $C_1$ and $C_2$ depend on $A,B$,
and $\eta$ is the critical value of the Green function $G$
of the region $\bC\backslash I$ with pole at infinity.
The arguments in \cite{F1} do not give optimal values of $C_1,C_2$.

When $A=B$, we have $e^{-\eta}=\sqrt{(A-1)/(A+1)}$, and the result
obtained in \cite{EY} is
\begin{equation}\label{bbb}
L_{2m+2}=L_{2m+1}\sim
\frac{\sqrt{2}(A-1)}{\sqrt{\pi A}}(2m+1)^{-1/2}
\left(\frac{A-1}{A+1}\right)^m.
\end{equation}
In this paper we will obtain a result of the same precision
for arbitrary $A$ and $B$.
In the case $A\neq B$, the ratio $\sqrt{n}e^{n\eta}L_n$ oscillates.
Similar oscillating asymptotic behavior was found by Akhiezer
for the polynomials of least deviation from $0$,
that is, for the error of the best
uniform approximation of $x^{n+1}$ by polynomials
of degree at most $n$ on two intervals.

To state our main asymptotic result we introduce certain characteristics
of the region $\bbC\backslash I$.
Let
\begin{equation*}
G(x)=G(x,\infty)= \int^x_{-1}\frac{C-x}{\sqrt{(1-x^2)(x+A)(B-x)}}dx,\quad -1<x<1,
\end{equation*}
be the Green function
of $\bC\backslash I$ with pole at infinity,
see, for example, \cite{A3},
where $C\in(-1,1)$ is the unique critical point,
\begin{equation*}
C=\frac{\int_{-1}^1((1-x^2)(x+A)(B-x))^{-1/2}xdx}{
\int_{-1}^1((1-x^2)(x+A)(B-x))^{-1/2}dx}.
\end{equation*}
We introduce
positive constants $\eta=G(C,\infty)$ and
\begin{equation*}
\eta_1=-\frac 1 2 G^{\prime\prime}(C)=\frac{1}{2\sqrt{(1-C^2)(C+A)(B-C)}}.
\end{equation*}
The Green function $G(z,C)$ satisfies
$$G(z,C)=-\ln|z-C|+\eta_2+O(z-C),\quad z\to C,$$
and this relation defines the Robin constant $\eta_2.$

Let $\omega(x)=\omega(x,[-A,-1],\bC\backslash I)$ be the harmonic
measure of the interval $[-A,-1]$.
An explicit formula for $\omega$ is
\begin{equation}\label{harmeasure}
\omega(z)=\Im\frac{\int_{-1}^z((x^2-1)(x+A)(B-x))^{-1/2}dx}{
\int_{-1}^1((x^2-1)(x+A)(B-x))^{-1/2}dx}.
\end{equation}
In our notation related to  theta functions
we follow Akhiezer's book \cite{A3}.

\begin{theorem}
The error $L_n$ of the best polynomial approximation
of $\sgn(x)$ on $I=[-A,-1]\cup[1,B]$ satisfies
\begin{equation}\label{11a}
\displaystyle
L_n=(c+o(1))n^{-1/2}e^{-n\eta}
\left|
\frac{\vartheta_0\left(\frac{1}{2}(\{ n\omega(\infty)+\omega(C)\}
-\omega(C))|\,\tau\right)}
{\vartheta_0\left(\frac{1}{2}(\{ n\omega(\infty)+\omega(C)\}
+\omega(C))|\,\tau\right)}\right|,
\end{equation}
where
$$c=2(\pi\eta_1)^{-1/2}e^{-\eta_2},$$
\begin{equation}\label{aaa}
\displaystyle\tau=i\frac{\int_1^B((t^2-1)(B-t)(A+t))^{-1/2}dt}{
\int_{-1}^{1}((1-t^2)(B-t)(A+t))^{-1/2}dt},
\end{equation}
and
$$\vartheta_0(t|\tau)=1-2h\cos2\pi t+2h^4\cos4\pi t-2h^9\cos6\pi t+\ldots,
\quad
h=e^{\pi i\tau},$$
is the theta-function. In (\ref{11a}) we used the notation
$\{ x\}$ for the fractional part of $x$.
\end{theorem}

Our method is somewhat different from the methods of previous
authors. It is based on an exact representation of
the extremal polynomial
as a composition of conformal maps of explicitly described regions.
This can be considered as a development of the arguments in \cite{E,EY}.
Our representation of extremal polynomials permits to find
their asymptotic behavior in various regimes and their zero distribution.
Actually, the main asymptotic result of this paper is Theorem~7.1,
which has somewhat technical statement, and Theorem~1.1 is
a simple corollary.
For example, according to \cite{F3}, the numbers $n_1$ and $n_2$
of zeros of the extremal polynomial $P_n$ on $[-A,-1]$ and $[1,B]$,
respectively, satisfy
$$\lim_{n\to\infty}n_1/n=\omega(\infty)\quad\mbox{and}\quad
\lim_{n\to\infty}n_2/n=1-\omega(\infty),$$
while our Theorem~7.1 implies
a stronger conclusion: $n_1=n\omega(\infty)+O(1)$.

However, in this paper
we focus on the error term of the polynomial
approximation, and do not explore other corollaries
from Theorem~7.1.

A representation of extremal polynomials is described in sections 2, 3,
where we use an entire function introduced in \cite{EY}.
In Section 4 we find an integral representation of the principal
conformal map
involved, and then study its asymptotics in sections 5-7.
We derive (\ref{bbb}) as a special case
of Theorem~1.1 in Section~8.
Finally, in Section~9, we sketch without proof the limit case $B=\infty$.
In this case, instead of approximation by polynomials
one has to consider approximation by entire functions of order $1/2$,
normal type.

\section{Preliminaries}
We begin by recalling the construction of the entire
function $\tilde{S}(z,a)$ of exponential
type one which gives the best
uniform approximation of $\sgn(x)$ on the set
$(-\infty,-a]\cup[a,\infty)$, where $a>0$.
There is a unique such function for every $a>0$;
it is odd and satisfies
\begin{equation}\label{twin1}
\tilde S(a,a)=1-L(a),\ \tilde S(\tilde c_k,a)=1-(-1)^k L(a),
\end{equation}
where $L(a)$ is the approximation error, and
$\tilde{c}_1<\tilde{c}_2<\ldots$ the sequence
of positive critical points.
The graph of
this function is shown in Fig.~1.
\begin{figure}
\begin{center} \includegraphics[scale=0.8]
{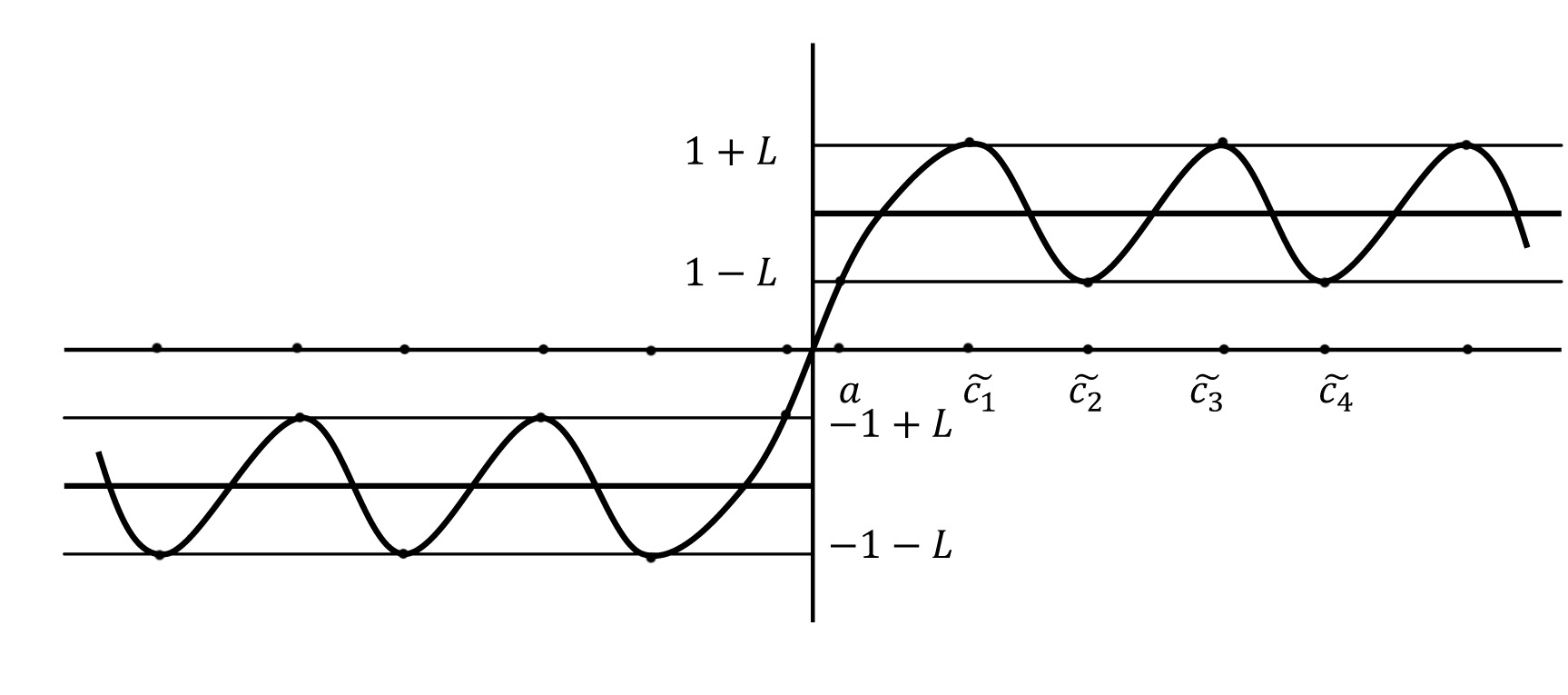}\end{center}
\caption{Graph of the function $\tilde S(z,a)$ on the real axis}\label{fig1}
\end{figure}
We define the positive number $b=b(a)$ by $\cosh b=1/L(a)$.
It is easy to see that $b$ is a continuous increasing
function of $a$, and the correspondence
$a\mapsto b$ is a homeomorphism of the positive ray
onto itself.

For every $b>0$, we consider the region
$$\Omega=\{ x+iy: x>0,y>0,x>\arccos(\cosh b/\cosh y)\;
\mbox{for}\; y>b\}.$$
This region is shown in Fig.~2; it consists
of the points in the first quadrant to the right
of the curve
$$\gamma_b:=\{\arccos(\cosh b/\cosh t)+it:b\leq t<\infty\}.$$
Let $\tilde{\psi}$ be the conformal map of the first
quadrant $\bbC_{++}$ onto $\Omega$, normalized by
\begin{equation}\label{twoinf3}
\tilde\psi(z)=z+\dots,\ z\to \infty\ \text{and}\ \tilde\psi(0)=b.
\end{equation}
Put $a=\tilde{\psi}^{-1}(0)$.
\begin{figure}
\begin{center} \includegraphics[scale=0.6]
{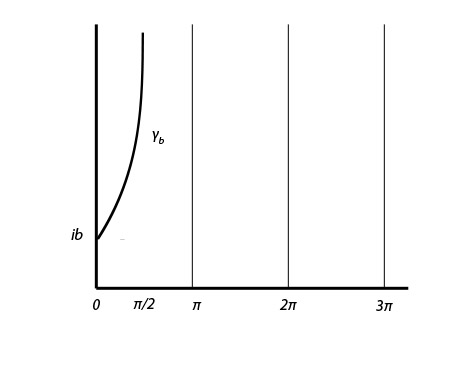}\end{center}
\caption{Domain $\Omega$ such that $\tilde S(z)=1-L(a)\cos\tilde \psi(z)$,
$\tilde\psi:\bbC_{++}\to \Omega$.}\label{fig2}
\end{figure}

In \cite{EY} we proved that
\begin{equation}\label{raz}
\tilde S(z,a)=1-L(a)\cos\tilde\psi(z),\quad
z\in\bbC_{++}:=\{ z: \Re z>0,\; \Im z>0\}.
\end{equation}
As the right hand side of (\ref{raz}) takes real values on the positive
ray and imaginary values on the positive imaginary ray, $\tilde{S}$ extends
to the whole plane as an odd entire function.

The following asymptotics hold
\begin{equation}\label{asmain0}
\lim_{a\to\infty}\sqrt{a} e^a L(a)=\sqrt{\frac 2 \pi}.
\end{equation}

Notice that
all  critical points $\{\pm \tilde c_k\}$ of
$\tilde S(z,a)$ are real, and $\tilde\psi(\tilde c_k)=\pi k$.

It is convenient to modify a little this
conformal mapping. We write
$S(z,a):=\tilde S(\sqrt{z^2+a^2},a)$,
where $z$ belongs to the upper half-plane $\bbC_+$ with
the slit $\{it: 0<t\le a\}$, see Fig.~\ref{fig3}. Function $S$ is
not entire, it is only defined in the upper half-plane.
\begin{figure}
\begin{center}
\includegraphics[scale=0.8]
{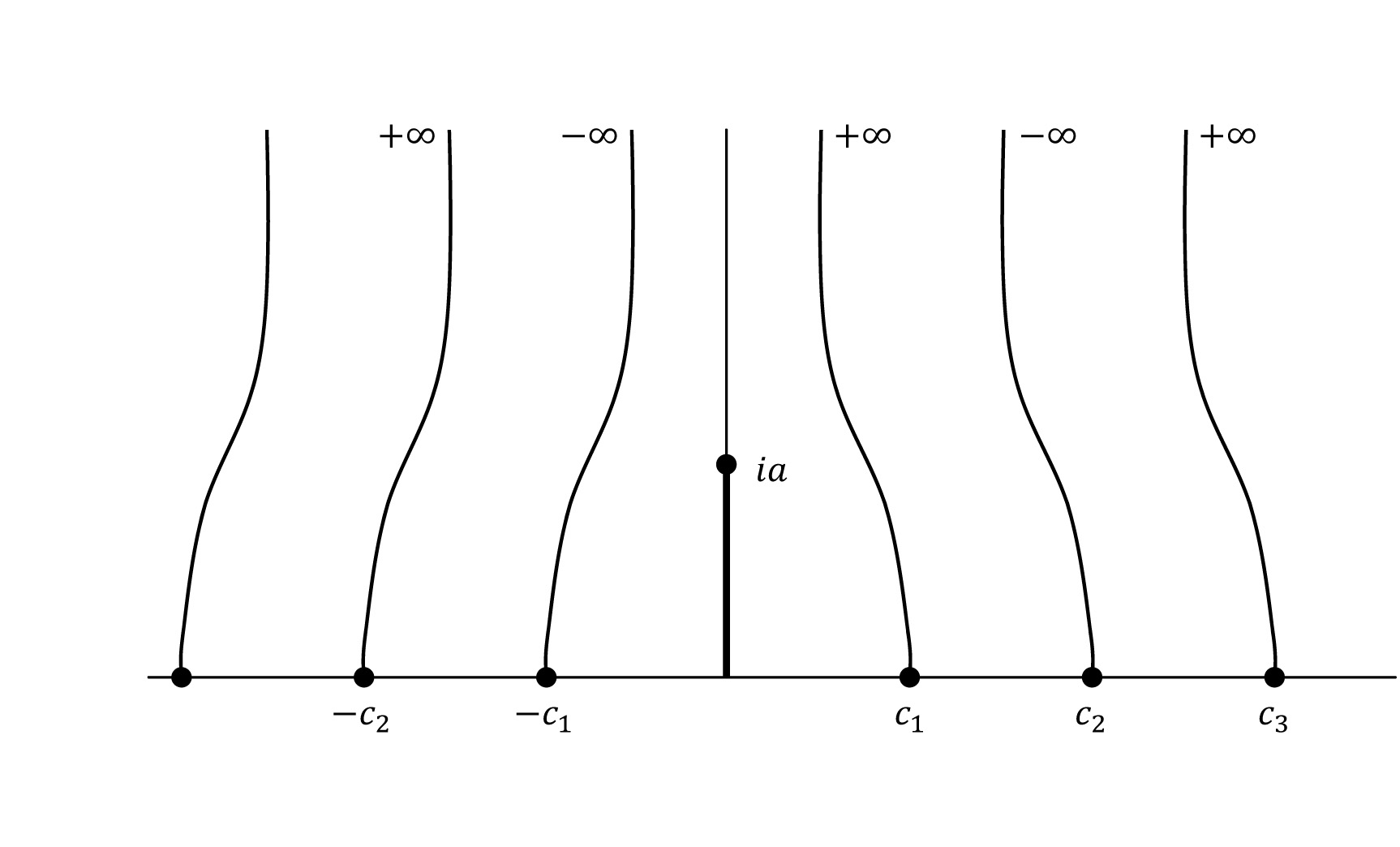}
\end{center}
\caption{Preimage $S^{-1}(\bbR,a)$  in the upper half-plane}\label{fig3}
\end{figure}

Again we have the conformal mapping
$\psi:\bbC_{++}\to\Omega$
but in contrast with \eqref{twoinf3}
\begin{equation}\label{twoinf5}
\psi(z)=z+\dots, \ z\to \infty,\ \ \psi(0)=0,
\end{equation}
and therefore $\psi(ia)=ib$.
The full preimage of the real
axis under $S$ in the upper half-plane consists
of the curves
\begin{equation}\label{ddd}
\delta_k=\psi^{-1}(\{\pi k+it:t>0\}),\quad k=\pm1,\pm2,\ldots .
\end{equation}
shown in Fig.~3.
These curves have vertical asymptotes $\{\Re z=\pi k-\pi/2\}.$

Now let $P_n(z)$ be the best approximation of
$\sgn(x)$ by polynomials
of degree at most
$n$ on two intervals $I=I_-\cup I_+$, $I_\pm\subset \bbR_\pm$.
Using a linear transformation we may always assume that
$I=[-A,-1]\cup[1,B]$.

Our goal is to obtain
a representation for the extremal polynomial
in the form of the composition
\begin{equation}\label{twin2}
P_n(z)=S(\Theta_n(z),a_n)
\end{equation}
where $\Theta_n$ is the conformal mapping\footnote{In what
follows, the letters $\Theta$ and $\theta$ are used
to denote conformal maps which have no relation
to theta-functions $\vartheta$.} of the upper
half-plane on a suitable ``curved" comb-like region,
and $a_n$ is an
appropriate value of the parameter~$a$.

First we give typical examples
of the representation \eqref{twin2}
and then show that these examples
exhaust all possibilities.
\vspace{.1in}

\noindent
{\em First example}.
For $n=4$, consider the following region
$\Pi_{2,3}$, see Fig. \ref{fig4}. Its boundary consists
of the vertical segment $[0,ia]$, the horizontal
segment $[-c_2,c_3]$ and the curves $\delta_{-2},\delta_3$.
\begin{figure}
\begin{center}
\includegraphics[scale=0.6]
{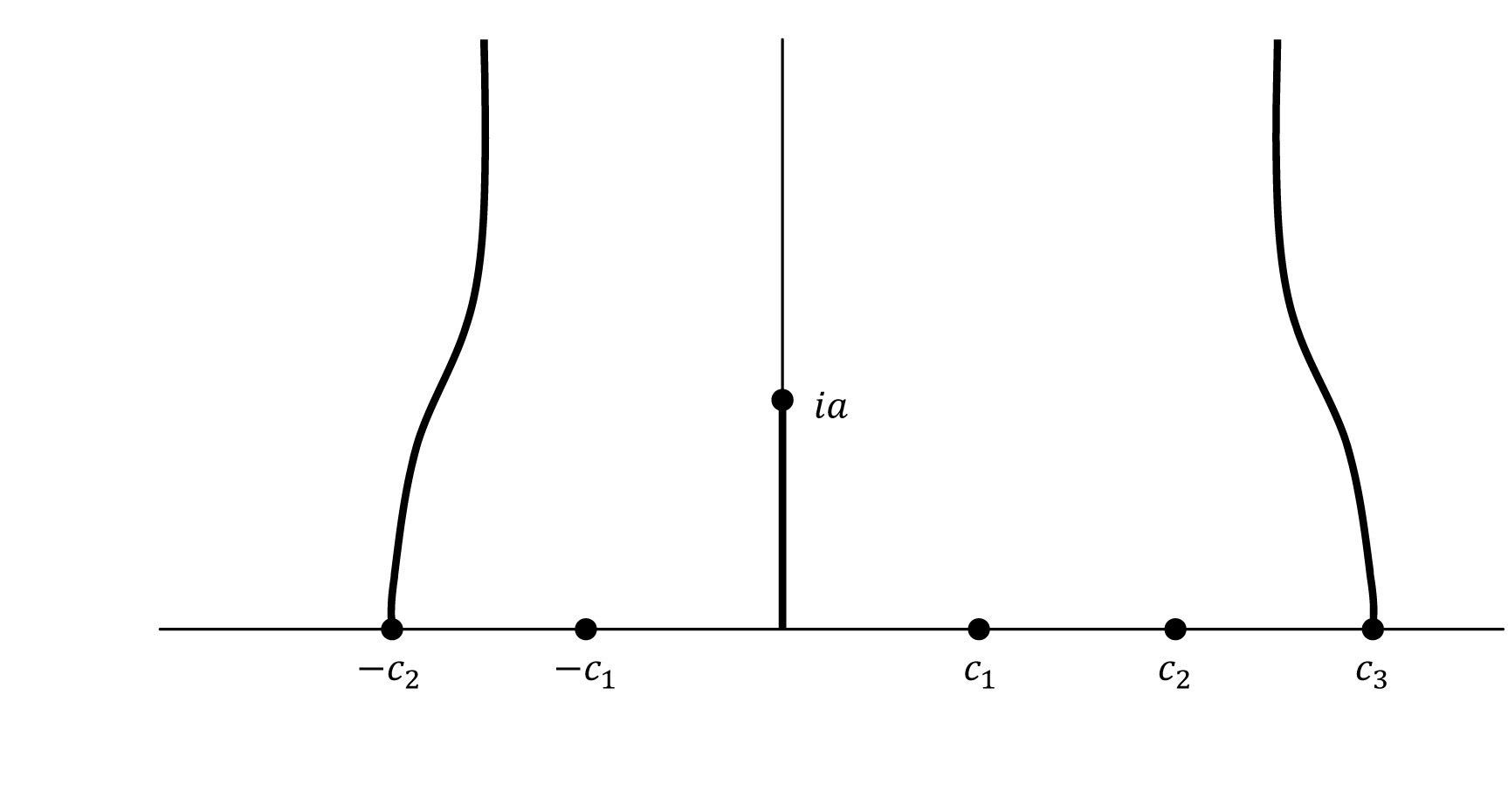}
\end{center}
\caption{Domain $\Pi_{2,3}$}\label{fig4}
\end{figure}

Let $\Theta(z)=\Theta_4(z)$ be the conformal mapping
\begin{equation}\label{twin3a}
\Theta: \bbC_+=\{z\in\bbC,\
\Im z>0\}\to \Pi_{2,3},\quad \Theta(\pm 1)= 0,\ \Theta(\infty)=\infty.
\end{equation}
The function $P(z)=S(\Theta(z),a)$ can be extended to the lower
half-plane due to the symmetry principle.
Therefore it is an entire function, which is, in fact,
a polynomial of degree $4$ due to its asymptotics at infinity.
The graph of this polynomial on the real axis is of the form given
in Fig. \ref{fig5},
\begin{figure}
\begin{center}
\includegraphics[scale=0.8]
{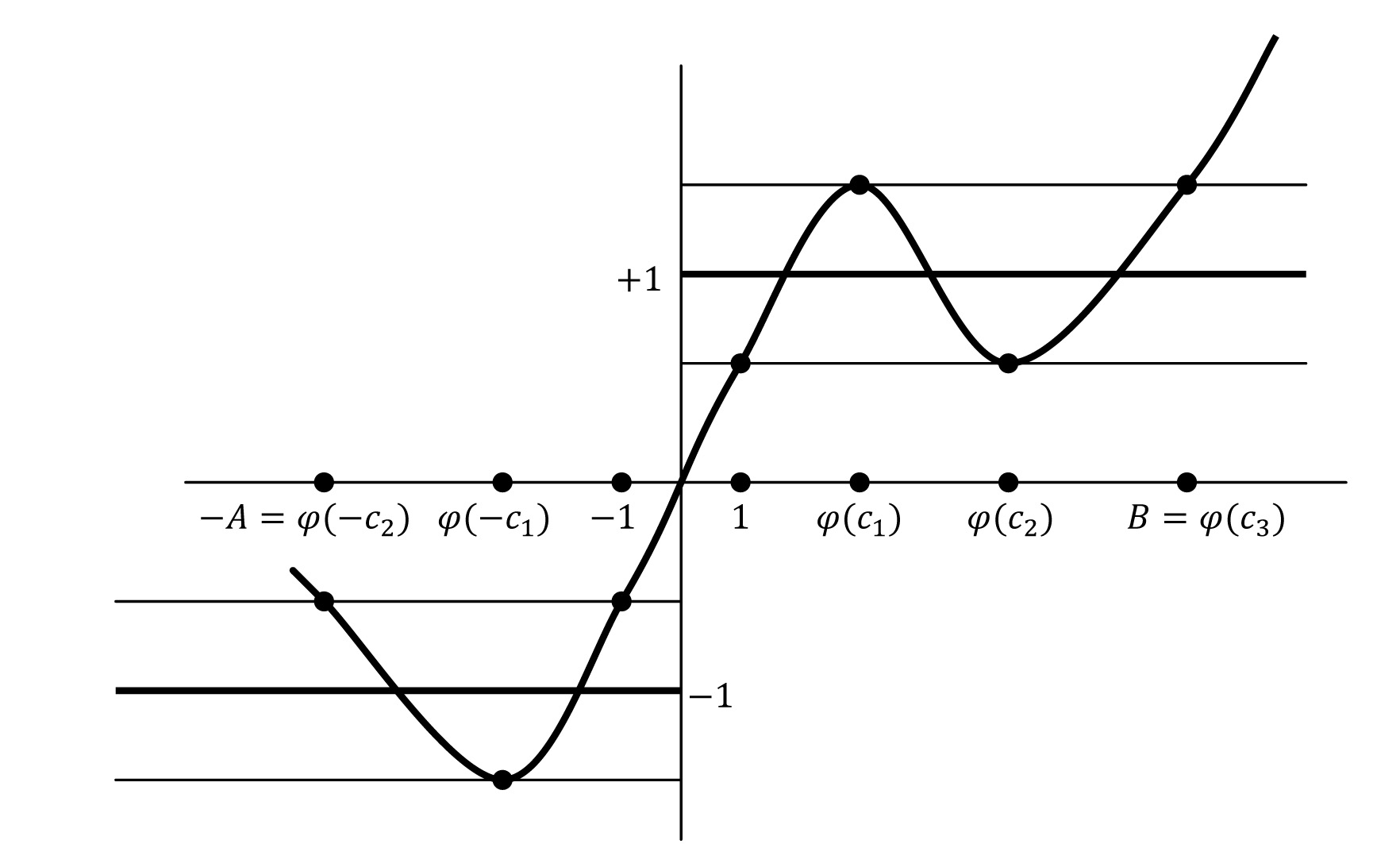}
\end{center}
\caption{Extremal polynomial corresponding to the region
$\Pi_{2,3}$}\label{fig5}
\end{figure}
where $\phi=\Theta^{-1}$. By the Chebyshev theorem
(for the two interval version of this theorem see \cite{A2}, \cite{A1}
\cite{A3}),
$P(z)$ is the extremal polynomial on the set $I$ with $A=-\phi(-c_2),$ and
$B=\phi(c_3)$.
\vspace{.1in}

\noindent
{\em Second example}. Let us point out that the
above polynomial $P(z)$ has $7$ points of alternance,
instead of $6$, which are required by the Chebyshev theorem for a polynomial of degree $4$. Therefore the same polynomial is extremal on the sets of two kinds
\begin{equation}\label{twin4}
I=[-A,-1]\cup[1,\phi(c_3)],\quad \phi(-c_2)<-A\le\phi(-c_1)
\end{equation}
or
\begin{equation}\label{twin5}
I=[\phi(-c_2),-1]\cup[1,B],\quad\phi(c_2)\le B\le\phi(c_3).
\end{equation}
\vspace{.1in}

\noindent
{\em Third example}. From the position
$I=[\phi(-c_2),-1]\cup[1,\phi(c_2)]$
we can start a deformation of the set $I$ and of the extremal polynomial.
Namely consider  the region $\Pi^+_{2,3}(h)$, see Fig. \ref{fig6}.
\begin{figure}
\begin{center}
\includegraphics[scale=0.8]
{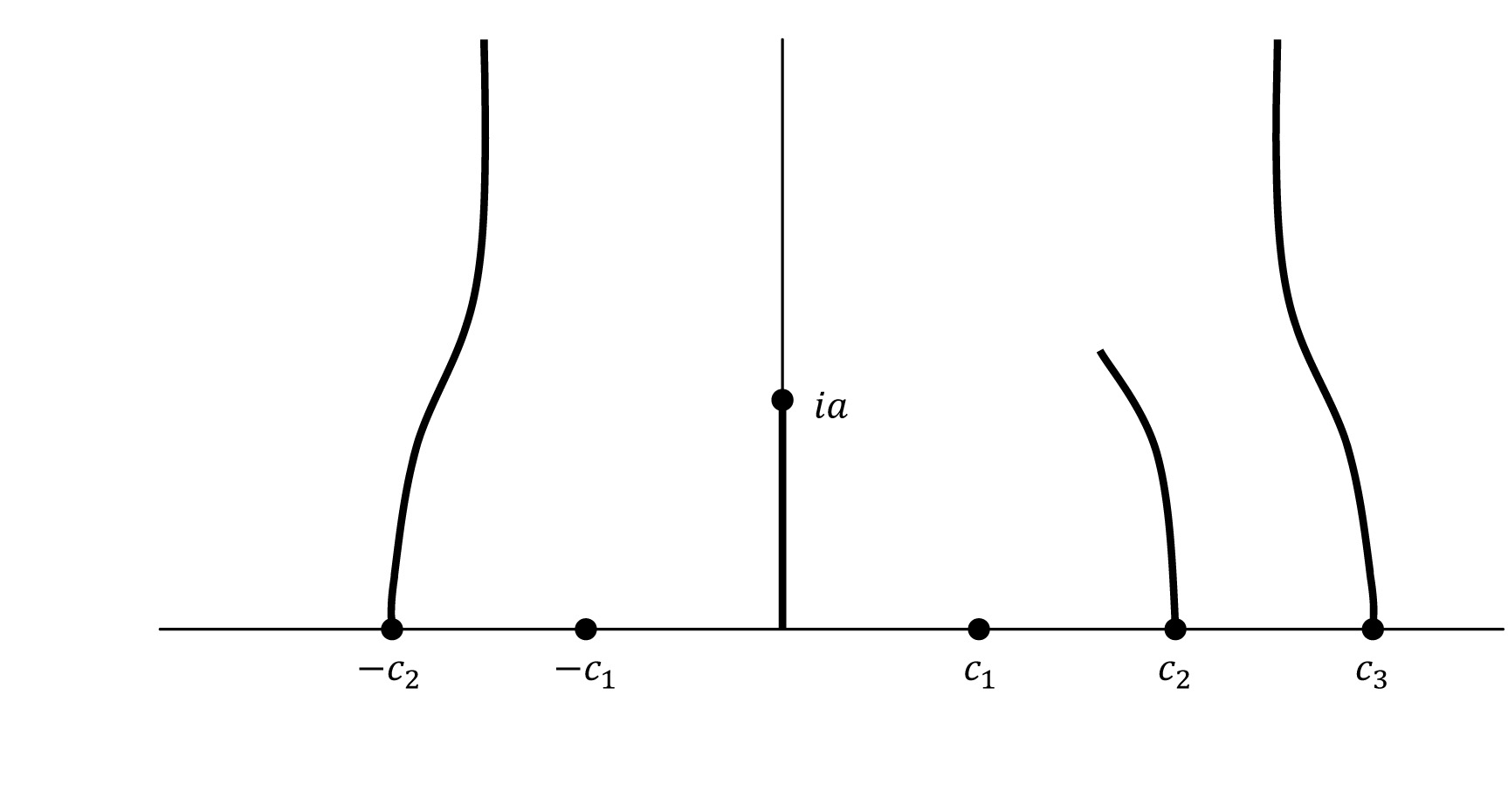}
\end{center}
\caption{Domain $\Pi^+_{2,3}(h)$}\label{fig6}
\end{figure}
Here we added to the boundary a segment of the curve $\delta_2$ that
starts at the critical point $c_2$ and has length
$h\in(0,\infty)$. In  this case
\begin{equation}\label{twin3}
\Theta: \bbC_+=\{z\in\bbC,\
\Im z>0\}\to \Pi^+_{2,3}(h),
\quad \Theta(\pm 1)=0,\ \Theta(\infty)=\infty,
\end{equation}
and $P(z):=S(\Theta(z),a)$ is of the form given in Fig. \ref{fig7}.
\begin{figure}
\begin{center}
\includegraphics[scale=0.8]
{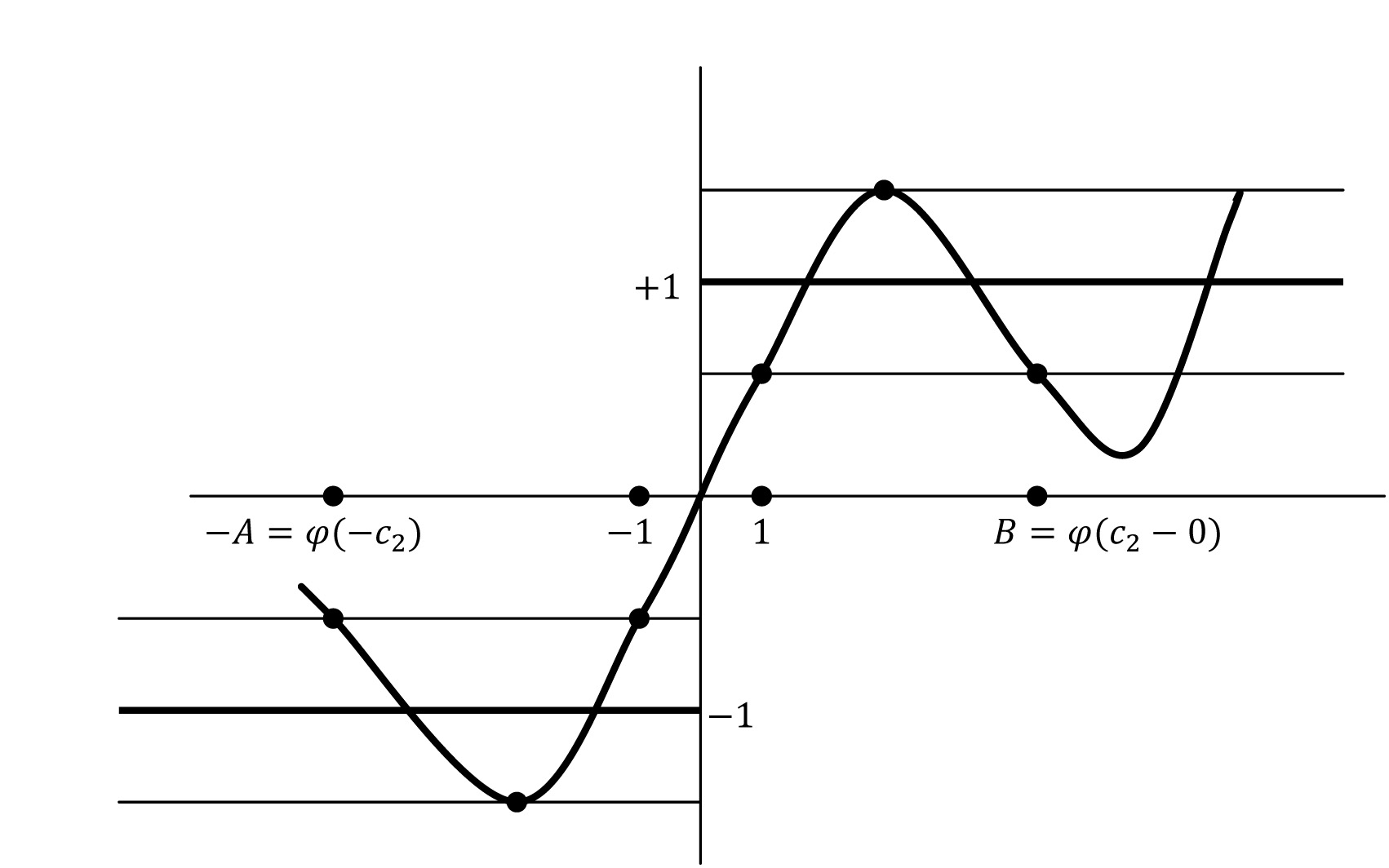}
\end{center}
\caption{ ...and the corresponding extremal polynomial}\label{fig7}
\end{figure}
For this new family of regions, which we
parametrized by positive $h$, the polynomial is extremal on the set
$I=[\phi(-c_2),-1]\cup[1,\phi(c_2-0)]$.

In the next section we show that these examples exhaust
all possibilities for the extremal polynomials.

\section{Parametrization}

We begin with a general description of
extremal polynomials.
Fix $a>0$. This defines the number $L=L(a)$ and the function $S(\cdot,a)$.
Let $k_1$ and $k_2$ be two positive integers.
Consider the region $\Pi_{k_1,k_2}$ in the upper half-plane
bounded by the curves $\delta_{-k_1}$ and $\delta_{k_2}$.
Then for $h\geq 0$ and $k_2\geq 2$, we define the region
$\Pi_{k_1,k_2}^+(h)$ by making in $\Pi_{k_1,k_2}$ a slit
along $\delta_{k_2-1}$ starting from $c_{k_2-1}$
and such that the length of its image
under $\psi$ is $h$. So $\Pi^\pm_{k_1,k_2}(0)=\Pi_{k_1,k_2}$.
Similarly we define
$\Pi_{k_1,k_2}^-(h)$ for $k_1\geq 2$ by making
a slit along $\delta_{-k_1+1}$.

Let $\Theta$ be the conformal map of the upper half-plane
onto $\Pi_{k_1,k_2}^+(h)$, normalized
by $\Theta(\pm1)=0,\;\Theta(\infty)=\infty$.
Consider the function
\begin{equation}\label{Aaa}
P(z)=S(\Theta(z),a).
\end{equation}
By construction, it is real on the real line, so the symmetry
principle implies that $P$ extends to an entire function.
By looking at the asymptotic behavior as $z\to\infty$
we conclude that $P$ is a polynomial of degree $k_1+k_2-1$.
All critical points of this polynomial are real.
If $h=0$, then all critical values are $-1\pm L$ and
$1\pm L$ on the negative and positive rays respectively.
If $h>0$, the extreme left critical value
is changed to $-1\pm L\cosh h$, or the extreme right
critical value is changed to $1\pm L\cosh h$.

We have seen in the previous section that each of these
polynomials $P$ is the extremal polynomial for some
$A$ and $B$. Now we prove that for every $A$ and $B$
one of these polynomials is extremal.

\begin{proposition} All extremal polynomials
are of the form (\ref{Aaa}) with $\Theta$ as defined above and
some $k_1, k_2, a$ and $h$.
\end{proposition}

We give an elementary proof of this proposition,
which is based on counting critical points and alternance points.
This proof does not extend to the case of entire functions,
so in Section~9 we will give another proof which is less elementary
but avoids counting.

\begin{proof}
In the proof we will use the following fact which is well-known
and easy to prove.
\vspace{.1in}

{\em Let $P_1$ and $P_2$ be two real polynomials
with all critical points real and simple, and suppose that their critical
points are listed in increasing order as
$c_1<c_2<\ldots<c_{n-1}$
and $c_1^\prime< c_2^\prime<\ldots
<c_{n-1}^\prime.$
If $P_1(c_j)=P_2(c_j^\prime)$ for $1\leq j\leq n-1$, then
$P_1(z)=P_2(cz+b)$ for some $c>0$ and real $b$.}
\vspace{.1in}

For a discussion and generalizations of this fact to entire functions, see
\cite{mclane}, \cite{vinberg}.

Let $A>1, B>1$, and a positive integer $n$ be given.
(We will deal with the degenerate case $A=1$ or $B=1$
later).
Let $P$ be the extremal polynomial of degree $n$ which exists
and is unique by Chebyshev's theorem. According to
Chebyshev's ``alternance theorem'', this polynomial $P$
is characterized by the following properties: let $Q(x)=P(x)-\sgn(x)$,
then
\begin{equation}
\label{A0}
|Q(x)|\leq L,\quad x\in[-A,-1]\cup[1,B],
\end{equation}
and there exist
\begin{equation}
\label{A1}
m\geq n+2
\end{equation}
points $x_1<x_2<\ldots<x_m$ in $[-A,-1]\cup[1,B]$
such that
\begin{equation}\label{A2}
|Q(x_j)|=L,\quad 1\leq j\leq m,\quad\mbox{and}\quad
Q(x_j)Q(x_{j+1})<0,\quad
1\leq j\leq m-1.
\end{equation}
These points $x_j$ are called the alternance points.
Evidently, all alternance points in $(-A,-1)\cup(1,B)$
are critical, that is, $P'(x)=0$ at all such points.
Let $K$ be the number of critical alternance points
and $N$ the number of non-critical alternance points.
We have the evident inequalities
$$K\leq n-1\quad\mbox{and}\quad N\leq 4.$$
Combined with (\ref{A1}) this gives
$$n+2\leq m=K+N\leq n+3.$$
So we have three possibilities:

a) $m=n+3,\; N=4,\; K=n-1$. The last two equalities imply
that all critical points of $P$ are real and simple,
and each of them is an alternance point which belongs
to $(-A,-1)\cup(1,B)$.
All $4$ points $-A,-1,1,B$ are non-critical
alternance points.
So the graph has the shape shown in Fig. 5.

b) $m=n+2,\; N=3, \; K=n-1$.
Again all critical points are real, simple, belong to
$[-A,-1]\cup[1,B]$, and each critical point is
an alternance point. All endpoints $-A,-1,1,B$ except
possibly one are alternance points.
Let us show that $-1$ and $1$ {\em are} alternance points.

Proving this by contradiction, suppose,
for example that $-1$ is not an alternance point.
Then $1$ cannot be a critical point
because $N=3$.
Thus $P'(x)\neq 0$ on $[-1,1]$ and $P(1)\geq 1-L>-1+L
\geq P(-1)$,
we conclude that $P$ is strictly increasing on
an interval $(-1-\epsilon,1+\epsilon)$ for some
$\epsilon>0$. This implies that $1$
is also not an alternance
point, a contradiction.

Thus the polynomial $P$ is of the type described in
(\ref{twin4}), (\ref{twin5}).

c)  $m=n+2,\; K=n-2,\; N=4$. In this case we have exactly
one simple critical point $z$ which is not
an alternance point.
Evidently this exceptional critical
point is real. We claim that it belongs to
$\bbR\backslash[-A,B]$.

First of all, $z\notin\{-A,-1,1,B\}$ because $N=4$ so
all these $4$ points are non-critical.
Second, $z$ cannot be in the interior of
one of the intervals $(-A,-1)$ or $(1,B)$. Indeed,
if it is in the interior of one of these intervals,
consider the adjacent alternance
points $x_j$ and $x_{j+1}$ on the same interval
such that $x_{j}<z<x_{j+1}$.
Such $x_j$ and $x_{j+1}$ exist
because all endpoints of each interval are alternance points,
and $z$ is not an alternance point.
As $z$ is the unique
critical point on $(x_j,x_{j+1})$, we obtain a contradiction
with the alternance condition (\ref{A2}).
Finally we prove that $z\notin (-1,1)$,
Proving this by contradiction, suppose that $z\in(-1,1)$.
As $-1$ is an alternance point we have $P(-1)=-1\pm L$.
Suppose first that
\begin{equation}
\label{A5}
P(-1)=-1-L.
\end{equation}
 Then $P'(-1)< 0$ because
$-1$ is not critical ($N=4$ in the case we consider now),
and
(\ref{A0}) implies that $P'(-1)\leq 0$.
As $P'$ changes sign exactly once on $(-1,1)$ and the
point $1$ is also non-critical, we conclude that
$P'(1)>0$. As $P(1)=1\pm L$, (\ref{A0}) implies
$P(1)=1-L$. This equality and (\ref{A5}) contradict
the alternance condition (\ref{A2}).
The case $P(-1)=-1+L$ is considered similarly.

Let $c$ be the critical point which is outside $[A,B]$.
It is easy to see that $|P(c)+1|>L$ if $c<0$ and $|P(c)-1|>L$
if $c>0$. This is because $N=4$ and thus $A$ and $B$
are non-critical alternance points.

So in the case c) we have the graph of the type shown
in Fig 7.

To summarize, we proved that in all cases the critical
points are real and simple, all critical values, with at most
one exception are $-1\pm L$ on the negative ray and
$1\pm L$ on the positive ray, and the exceptional critical
value, if it exists, corresponds to an extreme (left or right)
critical point. If the exceptional critical point $c$
is positive then $|P(c)-1|>L$ and if $c$ is negative then
$|P(c)+1|>L$.

Polynomials $S(\Theta,a)$ constructed above permit to match
any such critical value pattern, so we conclude that
$P(z)=S(\Theta(cz+b),a)$ with $c>0$ and $b\in\bbR$.
Finally, the points $-1,1$ are always non-critical alternance points,
and this implies that $c=1$ and $b=0$.

It remains to consider the degenerate case.
Suppose, for example that $B=1$.
Then only $3$ alternance
points can be non-critical, so we are in the case b).
The extremal polynomial in this degenerate case can be
easily written explicitly:
$$P_n(x)=L_nT_n\left(\frac{2x+A+1}{A-1}\right)-1,$$
where $T_n(x)=\cos n\arccos x$, and
$$L_n=\frac{2}{T_n(1+4/(A-1))+1}\sim
4\exp\left(- n\ch^{-1}\left(1+\frac{4}{A-1}\right)\right)$$
is the approximation error.
\end{proof}
\medskip

\noindent \textbf{Remarks}.
It is easy to see that our polynomials depend continuously
on $h$.
When $h\to\infty$, we have $\Pi_{k_1,k_2}^+(h)\to\Pi_{k_1,k_2-1}$
and $\Pi_{k_1,k_2}^-(h)\to\Pi_{k_1-1,k_2}$ in the sense of
Caratheodory, and the corresponding polynomials converge
uniformly on compact subsets of the plane.

Let us show that $A(h)$ and $B(h)$
depend monotonically on $h$.
Let $h>h_1$, $\Theta(z)=\Theta(z,h)$
and $\Theta_1(z)=\Theta(z,h_1).$ Then the function
$w(z)=\Theta_1^{-1}\circ\Theta(z)$
maps the upper half-plane into itself,
and has the properties: $w(\pm1)=\pm1, w(\infty)=\infty.$
Thus it has a
representation
\begin{equation*}
    w(z)=\rho_\infty(z-1)+1+\int\frac{z-1}{(x-1)(x-z)}d\rho(x)
\end{equation*}
where $d\rho$ is positive and supported
on a compact set $I$ such that $x>B$ for all $x\in I$, and $\rho_\infty>0$.
Here we used $w(1)=1$ and $w(\infty)=\infty$.
Now we use the condition $w(-1)=-1$. We obtain
\begin{equation}\label{1}
\rho_\infty=1-\int\frac{1}{(x-1)(x+1)}d\rho(x)
\end{equation}

Since $w(B)=B_1$ we have
\begin{equation*}
B_1-B=(\rho_\infty-1)(B-1)+(B-1)\int\frac{1}{(x-1)(x-B)}d\rho(x).
\end{equation*}
Using \eqref{1} we get
\begin{equation*}
    B_1-B=(B-1)\int\frac{B+1}{(x-1)(x-B)(x+1)}d\rho(x)>0.
\end{equation*}
Similarly
\begin{equation*}
-A_1+A=(A+1)\int\frac{A-1}{(x-1)(x+A)(x+1)}d\rho(x)>0.
\end{equation*}

\section{Integral representations}

Asymptotic relations for the extremal polynomials
are based on an integral representation of the conformal map $\Theta$.

Consider the
conformal map of $\bC\backslash I$ onto
on the annulus in Fig. \ref{fig10}.
\begin{figure}
\begin{center}
\includegraphics[scale=1.0]
{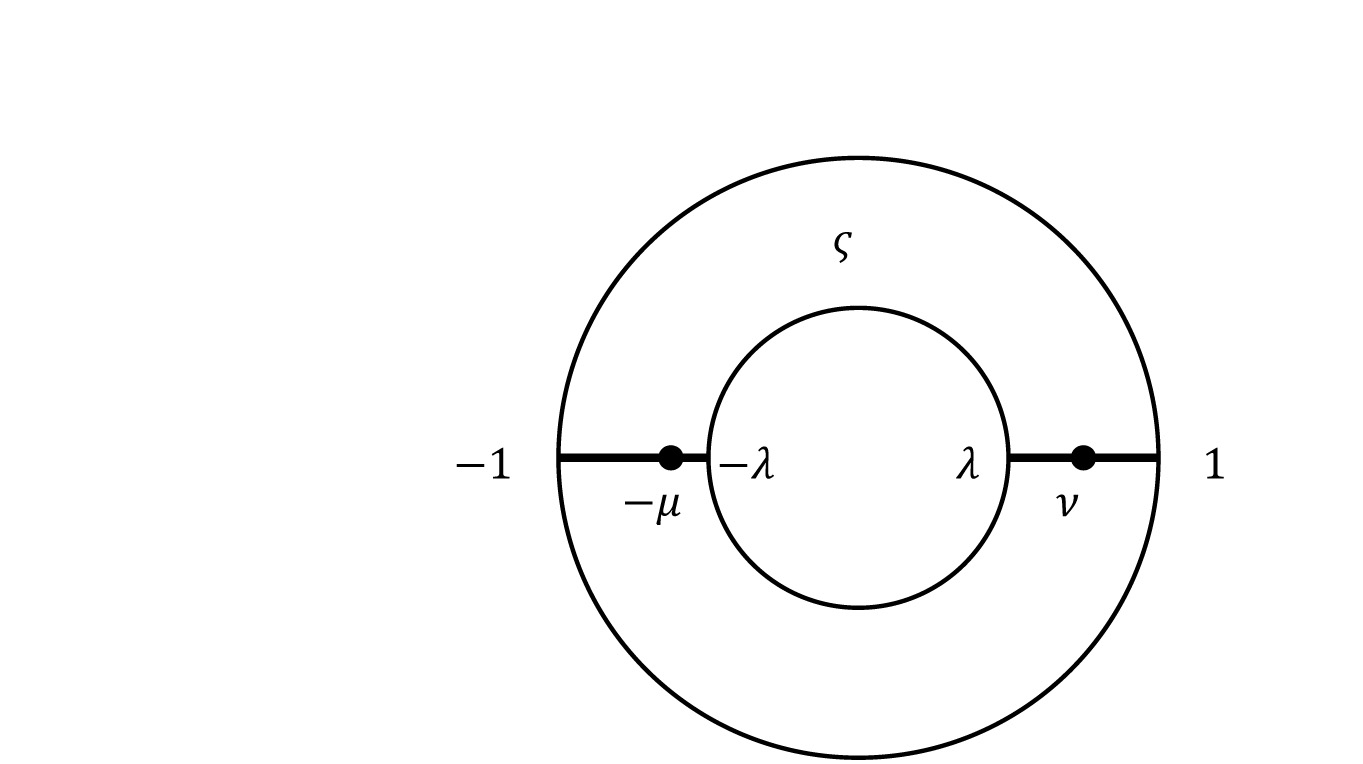}
\end{center}
\caption{The conformally equivalent annulus.}\label{fig10}
\end{figure}
Here we assume that the upper half-plane is
mapped on the upper part of the annulus with the
following boundary correspondence
\begin{equation}\label{efu2}
B\mapsto -1,\ -A\mapsto -\lambda,\ -1\mapsto\lambda,\
1\mapsto 1.
\end{equation}

By $G(z,z_0)$ we denote the (real)
Green function of the region $\bar\bbC\setminus I$,
$I=[-A,-1]\cup[1,B]$, with pole at $z_0$.
In particular $G(z):=G(z,\infty)$.
Recall that in the upper half-plane $G(z)$ can be represented
as the imaginary part of the conformal mapping $\Phi(z)$
of the upper half-plane
onto the region $\Pi$, Fig \ref{fig8}:
\begin{equation}\label{greengamma}
G(z)=\Im\Phi(z),\ \Im z>0,\ \Phi(\pm 1)=0,\
\Phi(\infty)=\infty.
\end{equation}
\begin{figure}
  \begin{center}
\includegraphics[scale=1]
{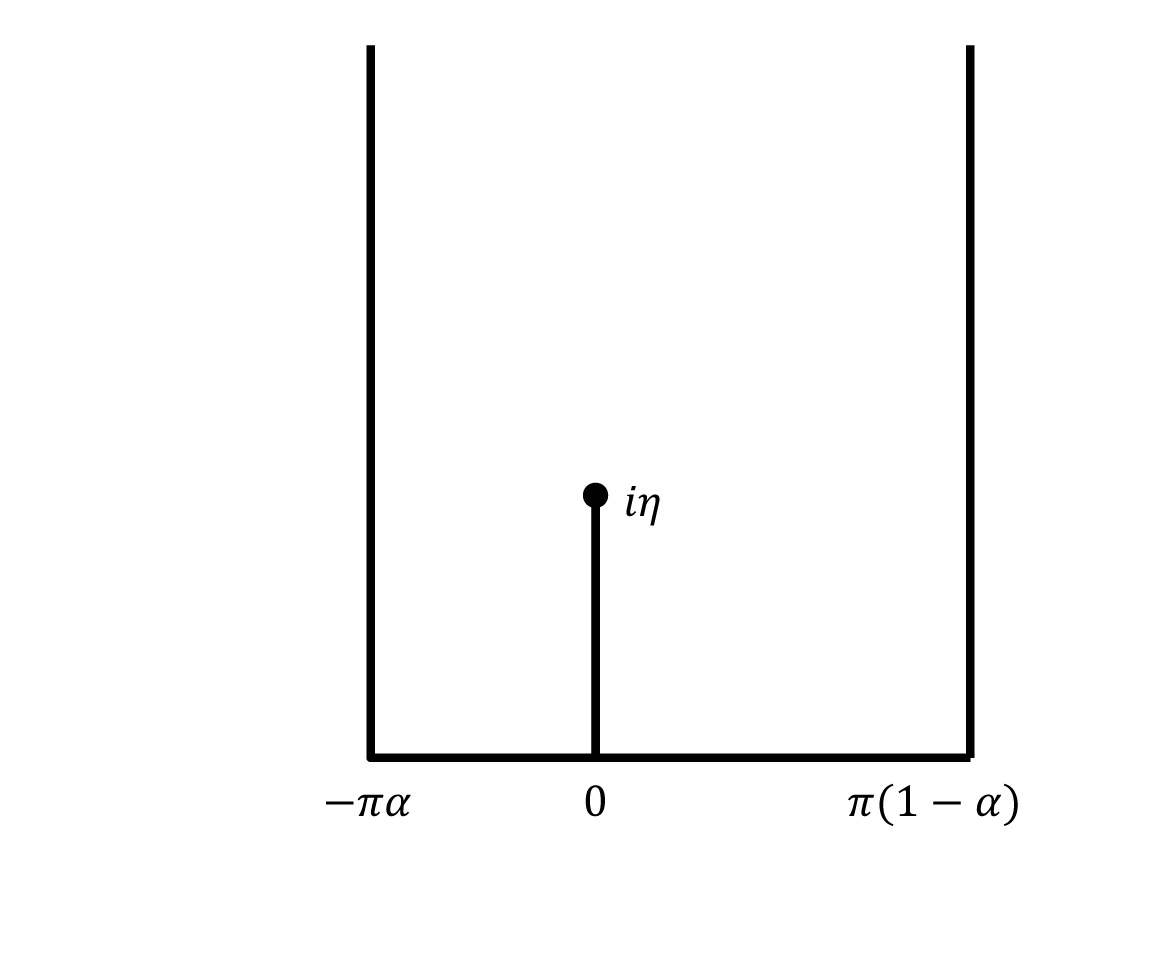}
\end{center}
\caption{The image $\Pi$ of the map $\Phi$.}\label{fig8}
\end{figure}

The map $\Phi(z)$ defines certain important
characteristics of the region: the critical value
\begin{equation}\label{greengamma1}
\eta=G(C), \quad C\in[-1,1]\quad\text{such that}
\quad\nabla G (C)=0,
\end{equation}
and the harmonic measure $\omega(z)$ of
the interval $[-A,-1]$. We have
\begin{equation}\label{greengamma2}
\Phi(-A)=-\pi\alpha,\quad\mbox{where}\quad \alpha=\omega(\infty).
\end{equation}

Now we associate to $\Phi(z)$ the  function
\begin{equation}\label{efu3}
g(\zeta)=-i\Phi(z(\zeta))
\end{equation}
where $\zeta$ belongs to the upper half of the annulus,
see Fig.~\ref{fig10}. This function can be extended
to the upper half-plane by the symmetry principle.
We have $G(z(\zeta))=\Re g(\zeta)$.
We call $g$ the {\em complex Green function}.

Let $-\mu\in[-1,-\lambda]$ corresponds to the infinity
in the $z$-plane, $-\mu=\zeta(\infty)$ (see Fig.~\ref{fig10}).
We define the jump function
\begin{equation}\label{jump}
j(\xi)=\begin{cases}1,&\xi\in(-1,-\mu)
    \\
    0,&\xi\in(-\mu,-\lambda)
    \end{cases}
\end{equation}
which we extend by the symmetry $j(1/\xi)=j(\xi)$,
$j(\lambda^2\xi)=j(\xi)$ on the whole negative ray.
Since
$$\Im g(\xi)=\left\{\begin{array}{ll}0,&\xi>0,\\
\pi\alpha,&\xi\in(-\mu,-\lambda),\\
\pi(\alpha-1),&\xi\in(-1,-\mu),
\end{array}\right.$$
we obtain the following
integral representation for $g(\zeta)$ in the upper half-plane
\begin{equation}\label{efu4}
g(\zeta)=
\int_{-\infty}^0\left\{\frac{1}{\xi-\zeta}-\frac{1}{\xi-1}\right\}
(\alpha -j(\xi))d\xi.
\end{equation}

\noindent
{\bf Remark.}
In the representation (\ref{efu4})
the normalization condition $\Phi(1)=0$ was used.
The second normalization condition $\Phi(-1)=0$ gives
$$\displaystyle\alpha=\overline{j}:=
\frac{\int_{-\infty}^0\left\{\frac{1}{\xi-\lambda}-\frac{1}{\xi-1}\right\}
j(\xi)d\xi}
{\int_{-\infty}^0\left\{\frac{1}{\xi-\lambda}-\frac{1}{\xi-1}\right\}d\xi}.$$
In what follows we will use the bar over a function to denote
similar averages.
\vspace{.1in}

Naturally we can simplify \eqref{efu4},
but the point is that we can write a similar representation
for the conformal mapping $\Theta_n(z)$.
Recall that for a given $n$, there exists a unique region
$\Pi(n)=\Pi^{\pm}_{k_1(n),k_2(n)}(h_n)$,  see Fig.~\ref{fig6},
such that the conformal mapping
$\Theta_n:  \bbC_+\to \Pi(n)$
represents the extremal polynomial \eqref{twin2}.
We define the function
\begin{equation}\label{efu7}
\theta_n(\zeta)=-i\Theta_n(z(\zeta)),\quad \theta_n(\lambda^2\zeta)=\theta_n(\zeta).
\end{equation}

We write the imaginary part of $\theta_n(\xi)$, $\xi<0$
as a sum
\begin{equation}\label{imtheta}
     \Im \theta_n(\xi)=\pi k_1(n)-\frac \pi 2+\pi nj(\xi)+\chi_n(\xi),
\end{equation}
so that $\chi_n(\xi)$ is a continuous function, which is normalized  by the condition $\chi_n(-\mu)=0$.
Then
\begin{equation}\label{eftheta}
\theta_n(\zeta)=
\frac 1 \pi \int_{-\infty}^0\left\{\frac{1}{\xi-\zeta}-\frac{1}{\xi-1}
\right\}(\pi k_1(n)-\frac \pi 2 -\pi nj(\xi)+\chi_n(\xi))d\xi.
\end{equation}

\begin{theorem} In the above notations
\begin{equation}\label{efu12}
\theta_n(\zeta)-n g(\zeta)
=\frac{1}\pi\int_{-\infty}^0\left\{\frac{1}{\xi-\zeta}-\frac{1}{\xi-1}
\right\}(\chi_n(\xi)-\bar\chi_n)d\xi,
\end{equation}
where
\begin{equation}\label{efu13}
\bar\chi_n=
\frac{\int_{-\infty}^0\left\{\frac{1}{\xi-\lambda}-\frac{1}{\xi-1}\right\}
\chi_n(\xi)d\xi}
{\int_{-\infty}^0\left\{\frac{1}{\xi-\lambda}-\frac{1}{\xi-1}\right\}d\xi}=
\pi n\alpha-\pi k_1(n)+\frac \pi 2.
\end{equation}
\end{theorem}

\begin{proof}
We subtract $ng(\zeta)$ in the form \eqref{efu4} from \eqref{efu12}. Then, we use the second normalization
condition $g(\lambda)=\theta_n(\lambda)=0$.
\end{proof}

This representation will imply Fuchs' asymptotics
as soon as we show that $\chi_n(\xi)$ is uniformly bounded.

\section{Fuchs' asymptotics}

Let us begin with a simple remark.
\begin{lemma}\label{lsimple}
Let $w(z)$ be a conformal mapping of the upper
half-plane onto a sub-region of the upper half-plane
which contains the half-plane $\Im w>\tau_0$.
Assume the normalization $w(z)\sim z$,
$z\to\infty$.
Then
\begin{equation}\label{estimatechi}
\Im w(z)-\Im z\in(0,\tau_0).
\end{equation}
\end{lemma}

\begin{proof}
Consider the integral representation of $\Im w(z)$
\begin{equation}\label{fu2}
\Im w(z)=\Im z+\frac 1\pi\int_{-\infty}^\infty
P(z,t)v(t)dt,
\end{equation}
where $P(z,t)$ is the Poisson kernel.
Since $0\le v(t)\le\tau_0$ we obtain the desired
inequality.
\end{proof}

\begin{proposition}
There are constants $C_1$ and $C_2$ (depending of the given  system of intervals $I$) such that
\begin{equation}\label{fuassnew}
    C_1\le a_n-n\eta\le C_2.
\end{equation}
\end{proposition}

\begin{proof}
Recall that the curves in Fig.~\ref{fig3}
were defined as preimages
of vertical lines in the region $\Omega$ in
Fig.~\ref{fig2}
under a conformal mapping
which maps the right half-plane into a sub-region
of the right half-plane.
Thus we can apply Lemma \ref{lsimple},
to obtain that $|\chi_n(\xi)|\le 2\pi$.  Therefore $|\chi_n(\xi)-\bar\chi_n|$ is also less than $2\pi$.

Now, $a_n$ is the maximum of $\theta_n(\xi)$ on the interval $(\lambda,1)$ and $\eta$ is the maximum of
$g(\xi)$ on the same interval. Due to the integral representation \eqref{efu12} the difference between these two functions is uniformly bounded in this interval. Thus \eqref{fuassnew} is proved.
\end{proof}

\begin{corollary}\label{cc1} The following limit relation holds
\begin{equation}\label{fu50}
\lim_{n\to\infty}\frac{\theta_n(\zeta)} {n}=g(\zeta),
\end{equation}
in particular
\begin{equation}\label{fu51}
\lim_{n\to\infty}\frac{k_1(n)} {n}=\alpha,\quad
\lim_{n\to\infty}\frac{\ln(1/L_n)}{n}=\lim_{n\to\infty}\frac{a_{n}}
{n}=\eta.
\end{equation}
\end{corollary}
\begin{proof}
We divide \eqref{efu12} by $n$ and pass to the limit.
\end{proof}

\begin{figure}
\begin{center}
\includegraphics[scale=0.8]{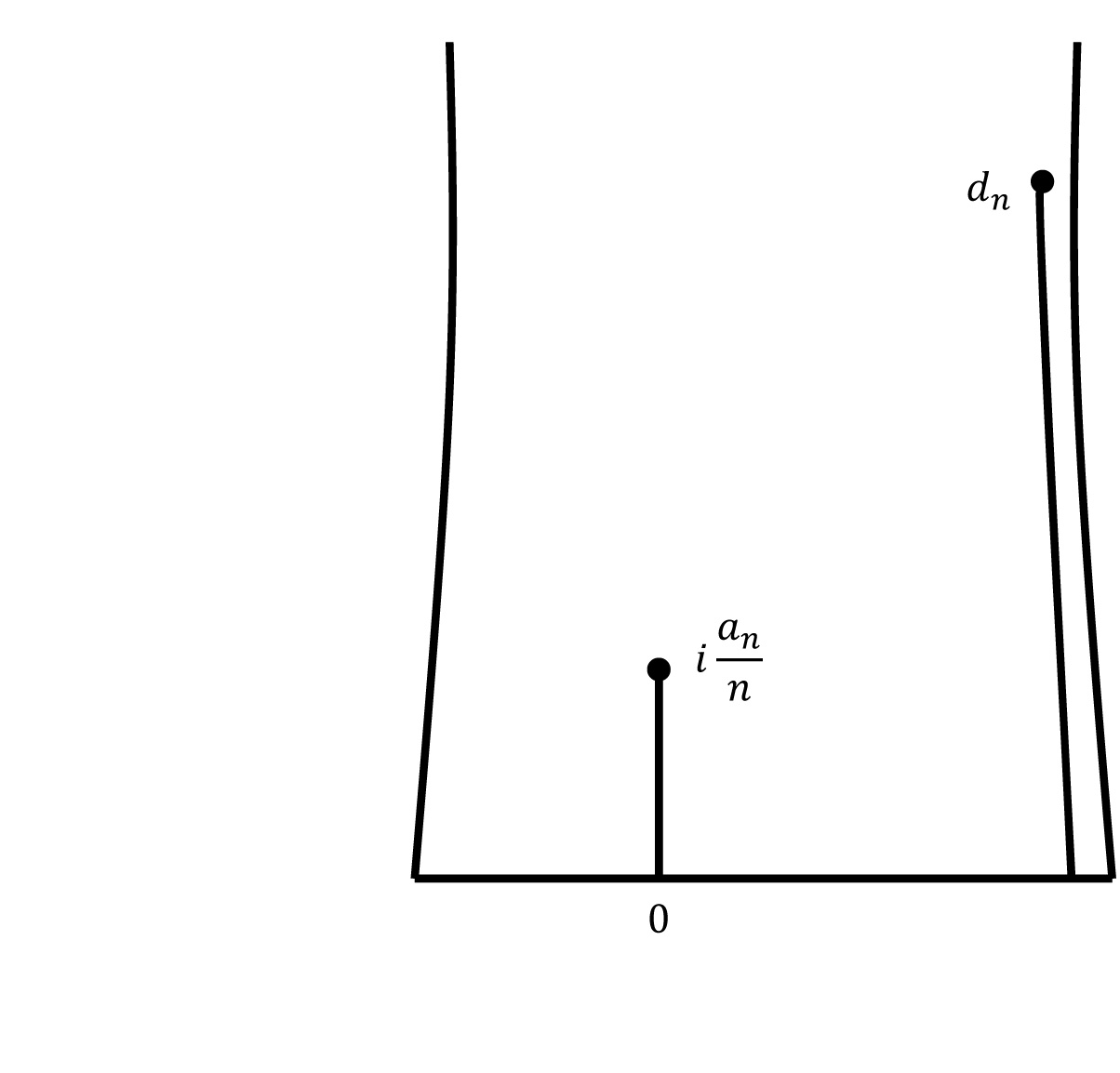}
\end{center}
\caption{The rescaled region $\Pi(n)/n$ for a large $n$.}
\label{figlinzdn}
\end{figure}

Corollary \ref{cc1} has the following geometric
interpretation.
Making the rescaling $\Pi(n)\to\Pi(n)/n$
we obtain the limit conformal mapping
onto the region shown in Fig. \ref{fig8}.
Let us look more carefully at the limit procedure,
see Fig \ref{figlinzdn}:
the distance between the additional cut and one of the
infinite cuts (left or right one) approaches zero,
however the position $d_n$ of the end point of the additional
cut influences the asymptotic behavior
along various subsequences $\{n_l\}$.
We define the subsequences by the condition:
there exists a
limit $d=d(\{n_l\})=\lim_{l\to\infty}d_{n_l}$.
Taking into account
this point $d$,
in the next section we describe the asymptotic
behavior of $L_n$ in a more precise way.

\section{The limit density $\chi(\xi)$}
We fixed a subsequence $\{n_l\}$
such that the limit $d=\lim_{l\to\infty}d_{n_l}$ exists.
Our main goal in this section is to show that the
limit density
\begin{equation}\label{ldn0}
\chi(\xi)=\lim_{l\to\infty}\chi_{n_l}(\xi)
\end{equation}
exists and to find this limit.

We start with the following general lemma.
Let $f$, be a bounded increasing differentiable
function defined for $x>0$, and suppose that
 $f(x)=0$ for $0<x\leq b$ with some $b>0$.
We consider the region
\begin{equation}\label{ldn1}
\tilde\Omega_f=\{z=x+iy: x>0,\; y>f(x)\},
\end{equation}
(it looks like the region
$\Omega$ in Fig. \ref{fig2} reflected in the line $x=y$).
Let $w$ be the conformal map
from the first quadrant $\bbC_{++}$ onto
$\tilde\Omega_f$ with the normalization
\begin{equation}\label{ldn2}
w(z)\sim z, \ z\to\infty,\quad w(0)=0.
\end{equation}
Let $a$ be the point such that $w(a)=b$.

\begin{lemma}\label{lemld1} Let $w(x)=u(x)+iv(x)$,
$x\ge a$. Then
$f(x)\le v(x)$.
\end{lemma}

\begin{proof} We extend $w$ by the symmetry principle
to the map of the upper half-plane into itself
(the extended map is still denoted by $w$),
and use the integral representation
\begin{equation}\label{ldn3}
w(z)=z+\frac 1 \pi \int_{a}^\infty\left\{\frac 1{t-z}-\frac 1{t+z}\right\}v(t)dt.
\end{equation}
For $x\ge a$ we have
\begin{equation}\label{ldn4}
w(x)=x+\frac 1 \pi \int_{a}^\infty\frac {2x}{t+x}\frac {v(t)-v(x)}{t-x}dt+\frac{v(x)}{\pi}
\ln\frac{x+a}{x-a}+iv(x).
\end{equation}
Therefore
\begin{equation}\label{ldn5}
u(x)=x+\frac 1 \pi \int_{a}^\infty\frac {2x}{t+x}\frac {v(t)-v(x)}{t-x}dt+\frac{v(x)}{\pi}
\ln\frac{x+a}{x-a}>x.
\end{equation}
Since $f(x)$ is increasing we obtain
\begin{equation}\label{ldn6}
f(x)<f(u(x))=v(x).
\end{equation}
\end{proof}

We apply Lemma \ref{lemld1} to obtain
the limit density for the conformal map of the first
quadrant onto
the region $\Omega$ in Fig.~2.
Namely, as before we consider the conformal map
$w(z)=-i\psi(-iz)$, where $\psi$ is defined in
(\ref{twoinf5}) and extended by symmetry to the
right half-plane,
and the integral representation \eqref{ldn3} for it.
Let us notice that in our case we have the exact
formula
\begin{equation}\label{ldn7}
f(x)=\arccos\frac{\cosh b}{\cosh x},\ x\ge b.
\end{equation}
Between the values $a$ and $b$ there is a one-to-one correspondence,
moreover $b\sim a+1/2\ln a$.
Thus we have the density $v(x)=v(x,a)$ in \eqref{ldn3}
as a function of the parameter $a$ and we are interested
in the limit density $\tilde v(x):=\lim_{a\to\infty}
v(ax,a)$.

\begin{lemma}\label{lemld2} The following limit exists
\begin{equation}\label{ldn8}
\tilde v(x):=\lim_{a\to\infty} v(ax,a)=\begin{cases} 0,
& x<1\\ \frac \pi 2,& x>1.
\end{cases}
\end{equation}
\end{lemma}

\begin{proof}
It is evident, that $\tilde v(x)=0$ for $x\in (0,1)$.
For $x>1$ we use Lemma~\ref{lemld1} and the
asymptotic relation between $a$ and $b$:
\begin{equation}\label{ldn9}
v(ax,a)\ge \arccos\frac{\cosh b}{\cosh ax}\sim
\arccos \frac{\sqrt{a}}{e^{(x-1)a}}.
\end{equation}
On the other hand  $v(ax,a)\le \pi/2$,
thus the lemma is proved.
\end{proof}

Now we are in position to evaluate the limit density \eqref{ldn0}.

\begin{theorem} Let $\{n_k\}$ be a subsequence
such that $\lim d_{n_k}=d$. Without loss of generality,
we assume that $\Re d=\pi(1-\alpha)$
(alternatively $\Re d=\pi\alpha$).
The relation
\begin{equation}\label{defd}
d=\Phi(D)=ig(-\kappa),
\end{equation}
uniquely defines
$D\in [B,\infty]$ and $-\kappa\in[-1,-\mu]$.
Then
\begin{equation}\label{ld6}
\chi(\xi)=\lim_{l\to\infty}\chi_{n_l}(\xi)=\begin{cases}
\displaystyle
\frac 1 2\int_{|t|<\eta}
\frac{\pi\alpha}{(t-y)^2+(\pi\alpha)^2}dt,& -\mu<\xi<-\lambda,
\\
\displaystyle
\frac 1 2\int_{|t|<\eta}
\frac{\pi(\alpha-1)}{(t-y)^2+(\pi(\alpha-1))^2}dt,&
 -\kappa<\xi<-\mu,\\
\displaystyle
\frac 1 2\int_{|t|<\eta}
\frac{\pi(\alpha-1)}{(t-y)^2+(\pi(\alpha-1))^2}dt+\pi,&
 -1<\xi<-\kappa,\\
\end{cases}
\end{equation}
where $y=\Re g(\xi)$.
\end{theorem}

\begin{figure}
  \begin{center}
\includegraphics[scale=0.7]
{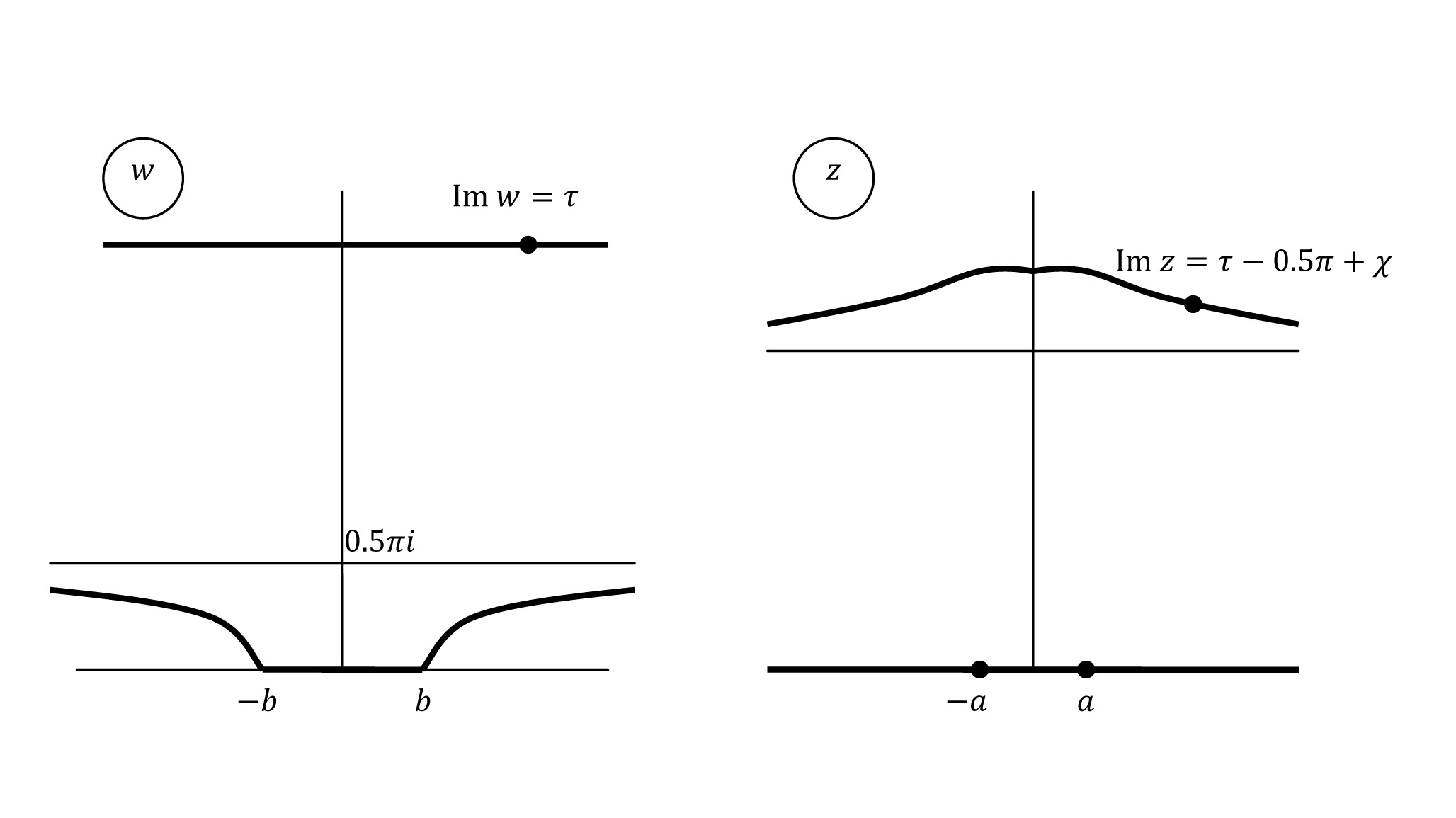}
\end{center}
  \caption{The preimage of the level line $\Im w=\tau$
}\label{last}
\end{figure}

\begin{proof} First we assume that $-\mu<\xi<-\lambda$. Let $z_l=\theta_{n_l}(\xi)$.
For a sufficiently large $l$, by \eqref{ldn3}, we have
\begin{equation}\label{ld2}
 \Im w_l= \Im z_l+\frac 1 \pi\int_{|t|>a_{n_l}}
\frac{\Im z_l}{(t-\Re z_l)^2+\Im z_l^2}v(t,a_{n_l})dt.
\end{equation}
Substituting
$\Im z_l=\pi k_1(n_l)-\frac{\pi}{2}+\chi_{n_l}(\xi)$
and $\Im w_l= \pi k_1(n_l)$ to (\ref{ld2}) (see Fig.~\ref{last}) we obtain
\begin{equation}\label{ld3}
\begin{split}\displaystyle
\pi/2-\chi_{n_l}(\xi)=\displaystyle &\frac 1\pi
\int_{|t|>a_{n_l}}
\frac{(\pi k_1(n_l)-\pi/2+\chi_{n_l}(\xi))v(t,a_{n_l})}
{(t-y_{l}(\xi))^2+(\pi k_1(n_l)-\pi/2+\chi_{n_l}(\xi))^2}dt\\ \displaystyle
 =&\frac 1\pi\int_{|t|>1}
\frac{(\pi k_1(n_l)-\pi/2 +  \chi_{n_l}(\xi))a_{n_l}v(a_{n_l}t,a_{n_l})}
{(a_{n_l}t- y_{l}(\xi))^2+(\pi k_1(n_l)-\pi/2+\chi_{n_l}(\xi))^2}dt
\end{split}
\end{equation}

By the leading term asymptotics,
Corollary \ref{cc1},  we have
\begin{equation}\label{ldn100}
\lim_{l\to\infty}\frac{k_1(n_l)}{n_l}=
\alpha,\ \lim_{l\to\infty}\frac{a_{n_l}}{n_l}=\eta,
\lim_{l\to\infty}\frac{y_l(\xi)}{n_l}=y(\xi).
\end{equation}
Passing to the limit in \eqref{ld3} we get
 \begin{equation}\label{ld4}
 \pi/2-\chi(\xi)=\frac 1\pi\int_{|t|>1} \frac {\eta\pi\alpha}{(\eta t-y)^2+(\pi\alpha)^2}\tilde v(t)dt,
\end{equation}
By Lemma~\ref{lemld2}, after trivial simplifications
we obtain the first equation in \eqref{ld6}.

In the second case $-\kappa<\xi<-\mu$,
for sufficiently large $l$, the point  $z_l$
corresponds to a point $w_l$ on the line
$\Im w_l=\pi k_2(n_l)$.
Thus we can repeat the previous arguments with
$\alpha$ replaced by $1-\alpha$
(let us mention that $\chi_{n_l}(\xi)$ is
negative here).

In the last case $\Im w_l=\pi (k_2(n_l)-1)$,
and this leads to the shift of the limit value by $\pi$.
\end{proof}

\section{Simplifying the result}

In this section  we prove the following theorem, which is our
main result, and which implies Theorem~1.1.

\begin{theorem}\label{th7.1} Let $\nu$ be the point in
the interval $(\lambda,1)$,
such that $g(\nu)=\eta$.
Fix a subsequence $\{n_l\}$ such that
$\lim_{l\to\infty}{d_{n_l}}=d=ig(-\kappa)$.
Let $g(\zeta,\nu)$ and $g(\zeta,-\kappa)$ be the
corresponding complex Green functions.
Then
\begin{equation}\label{result}
\lim_{l\to\infty}\{\theta_{n_l}(\zeta)-n_l g(\zeta)\}=
\frac 1 2 \ln\frac{\eta-g(\zeta)}{\eta+g(\zeta)}+g(\zeta,\nu)
-g(\zeta,-\kappa).
\end{equation}
\end{theorem}

\noindent
{\em Proof.}
First of all we split $\chi(\xi)$ into the sum of a continuous function $\chi_c(\xi)$ and the jump
\begin{equation}\label{jump2}
j_1(\xi)=\begin{cases}
1, &\xi\in(-1,-\kappa)\\
0, &\xi\in (-\kappa,-\lambda)
\end{cases}
\end{equation}
As usual the jump function is extended to the negative ray by the reflections $j_1(1/\xi)=j_1(\xi)$,
$j_1(\lambda^2\xi)=j_1(\xi)$.

Note that the jump function is related to the Green
function $G(z,D)$, compare \eqref{jump} and (\ref{efu4}).
Since
$$\Im g(\xi,-\kappa)=\left\{\begin{array}{ll}0,&\xi>0,\\
\pi\omega(D),&\xi\in(-\kappa,-\lambda),\\
\pi(\omega(D)-1),&\xi\in(-1,-\kappa),
\end{array}\right.$$
we obtain the following
integral representation for $g(\zeta,-\kappa)$ in the upper half-plane
\begin{equation}\label{jump3}
g(\zeta,-\kappa)=\int_{-\infty}^0\left\{\frac 1{\xi-\zeta}-
\frac 1{\xi-1}\right\}(\bar j_1-j_1(\xi))d\xi,
\end{equation}
where we use the notation $\bar j_1$ introduced in the Remark
in Section~4.
Notice that
$G(z(\zeta),D)=\Re g(\zeta,-\kappa)$, and
\begin{equation}\label{newxx}
\bar j_1=\omega(D).
\end{equation}

Due to the chosen normalization $\chi(-\mu)=0$, we have
\begin{equation}\label{chi2}
\chi(\xi)=\begin{cases}
\chi_c(\xi)+\pi j_1(\xi), &-\kappa<-\mu\\
\chi_c(\xi)+\pi j_1(\xi)-\pi, &-\mu<-\kappa
\end{cases}
\end{equation}

The main point is to evaluate the Cauchy transform
of the continuous part $\chi_c(\xi)$.

\begin{lemma}\label{ll} Let $\nu\in(\lambda,1)$
be such that $g(\nu)=\eta$, that is,
$\nu$ corresponds to the critical point
$C\in(-1,1)$. Let $g(\zeta,\nu)$ be the
corresponding complex Green function. Then
\begin{equation}\label{ae0}
\frac{1}\pi\int_{-\infty}^0
\left\{\frac{1}{\xi-\zeta}-\frac{1}{\xi-1}\right\}
(\chi_c(\xi)-\bar\chi_c)d\xi=
\frac 1 2 \ln\frac{\eta-g(\zeta)}{\eta+g(\zeta)}+g(\zeta,\nu).
\end{equation}
\end{lemma}

\begin{proof}
Using \eqref{ld6}, for $z=y(\xi)+i\alpha=g(\xi)$,
$\xi\in (-\mu,-\lambda)$, we have
\begin{equation}\label{ae1}
\chi_c(\xi)=\frac 1 2 \Im\int_{-\eta}^{\eta}
\frac 1{t-z} dt=\frac 1 2 \Im\ln\frac{\eta-z}{-\eta-z}.
\end{equation}

\begin{figure}
  \begin{center}
\includegraphics[scale=.8]
{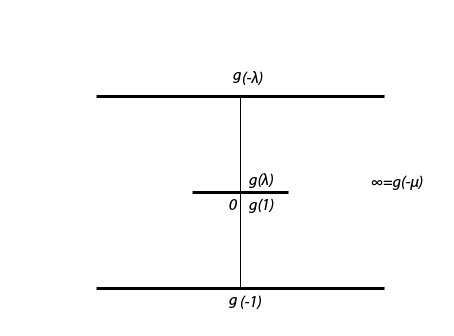}
\end{center}
  \caption{The image of the Green function}\label{fig13}
\end{figure}

So let us consider the function
\begin{equation}\label{ae2}
    f(\zeta):=\frac 1 2 \ln\frac{\eta-g(\zeta)}{-\eta-g(\zeta)}.
\end{equation}
The image of the function $g(\zeta)$ is shown in Fig. \ref{fig13}, the image of the fraction linear transformation is shown in Fig.  \ref{fig14}. Let us point out that for $\zeta$
in the upper half of the ring, Fig. \ref{fig10},
we obtain the values of $g(\zeta)$
in the right half-plane and for
$\displaystyle\frac{\eta-g(\zeta)}{-\eta-g(\zeta)}$
in the unit disk.
Thus,
\begin{equation}\label{ae3}
    \rho(\xi):=\Im f(\xi)=\begin{cases}\pi/2,& \lambda<\xi<\nu\\
    -\pi/2,&\nu<\xi<1
    \end{cases}
\end{equation}
here $\nu$ is such that $g(\nu)=\eta$.
\begin{figure}
  \begin{center}
\includegraphics[scale=0.8]
{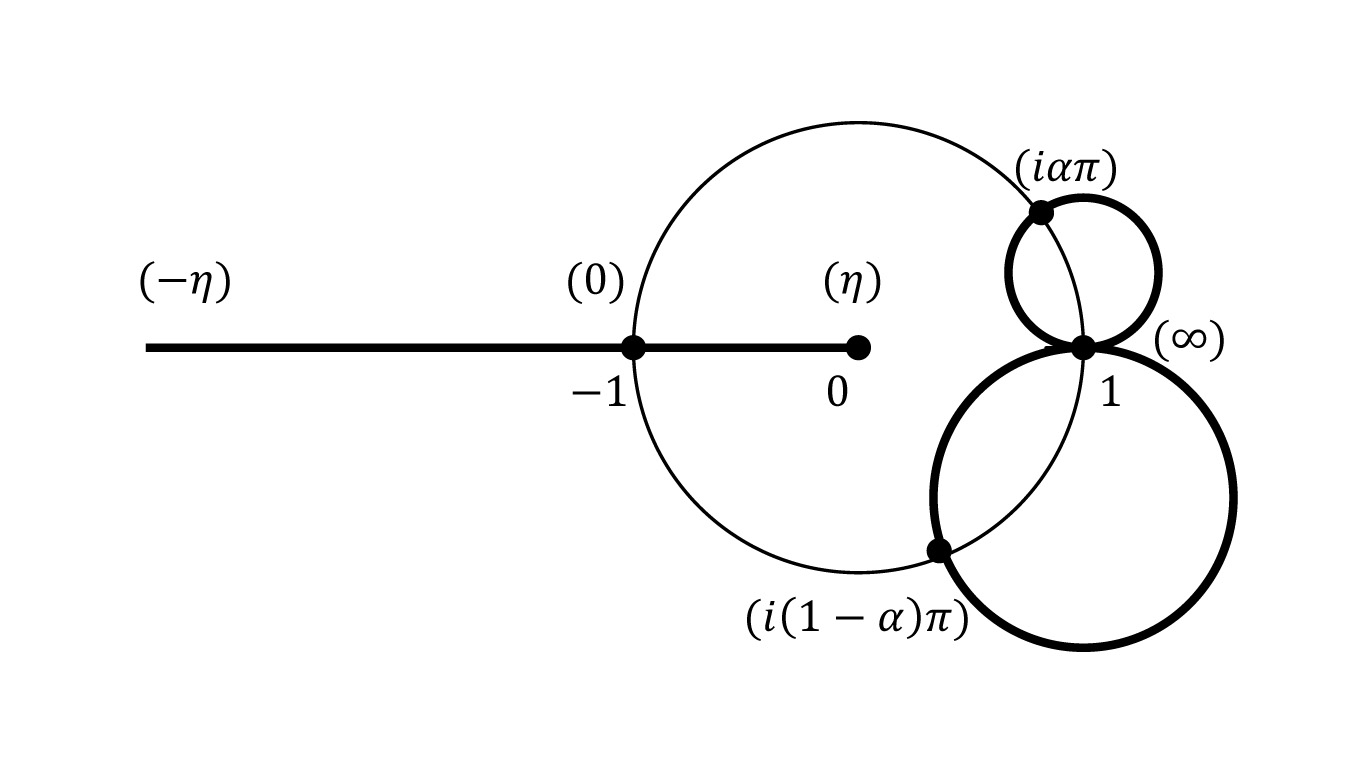}
\end{center}
  \caption{The image of $(\eta-g(\zeta))/(-\eta-g(\zeta))$}\label{fig14}
\end{figure}

We use the integral representation of $f(\zeta)+i\pi/2$
\begin{equation}\label{ae4}
   \frac 1 2 \ln\frac{\eta-g(\zeta)}{\eta+g(\zeta)}= f(\zeta)+\frac\pi 2 i=\frac 1\pi\int_{-\infty}^\infty\left\{\frac{1}{\xi-\zeta}-\frac{1}{\xi-1}\right\} \left(\rho(\xi)+\frac\pi 2\right)d\xi.
\end{equation}

The complex Green function related to the critical point $C\in (-1,1)$ we still normalized by the condition
$g(1,\nu)=0$.
Therefore
$$\Im g(\xi,\nu)=
\left\{\begin{array}{ll}-\pi,&\xi\in(\lambda,\nu),\\
-\pi(1-\omega(C)),&\xi<0,\\
0,&\xi\in(\nu,1).
\end{array}\right.$$
According to \eqref{ae3} we can represent
$g(\zeta,\nu)$ as
\begin{equation}\label{ae5}
\begin{split}
 g(\zeta,\nu)=&-(1-\omega(C))
 \int_{-\infty}^0\left\{\frac{1}{\xi-\zeta}-\frac{1}{\xi-1}\right\}d\xi\\
 -&
 \frac 1\pi\int_{0}^\infty\left\{\frac{1}{\xi-\zeta}-\frac{1}{\xi-1}\right\}\left(\rho(\xi)+\frac\pi 2\right)d\xi.
 \end{split}
\end{equation}

Recall that on the negative ray $\rho(\xi)=\chi_c(\xi)$, see (\ref{ae1}).
Adding \eqref{ae5} and \eqref{ae4}
we obtain \eqref{ae0},
moreover
\begin{equation}\label{newxxx}
\bar \chi_c=\pi(\frac 1 2 -\omega(C)).
\end{equation}
\end{proof}
Theorem \ref{th7.1} follows from Lemma~\ref{ll}, (\ref{efu12}) and
(\ref{jump3}).
\vspace{.1in}

\noindent
{\em Completion of the proof of Theorem~1.1}.
\vspace{.1in}

The error term $L_n$ satisfies
\begin{equation}\label{12a}
L_n\sim \sqrt{2/\pi}\, a_n^{-{1/2}}e^{-a_n},
\end{equation}
where $a_n=\max\{\theta_n(\xi):\lambda<\xi<1\}.$
This follows from (\ref{asmain0}) and our explicit
representation of the extremal polynomial (\ref{twin2}).

The necessary constants $C,\eta,\eta_1,\eta_2,\alpha$ which depend only on
$A$ and $B$,
and the harmonic measure $\omega(x)=\omega(x,[-A,-1],\bC\backslash I)$
were defined in the Introduction.

According to (\ref{newxx}), (\ref{newxxx}) and (\ref{chi2}) we have
\begin{equation}\label{chi22}
\frac{\bar \chi} \pi=\begin{cases}
1/2 -\omega(C)+\omega(D), &-\kappa<-\mu,\\
-1/2 -\omega(C)+\omega(D), &-\mu<-\kappa.
\end{cases}
\end{equation}
Notice that $\omega$ is a strictly increasing function on
$\bbR\backslash(-A,B)$ and the image of this set (together with the
infinite point) equals $[0,1]$. Therefore,
for every $n$ there exists a unique solution $D_n$ of the
equation
\begin{equation}\label{12c}
\omega(D_n)=\{\alpha n+\omega(C)\},
\end{equation}
where $\{\cdot\}$ stands for the fractional part.

Equations (\ref{efu13}), (\ref{chi22}) and
$d_{n_l}\to d$ imply that
\begin{equation}\label{newxxxx}
\omega(D_{n_l})\to \omega(D).
\end{equation}
Let $-\kappa_{n}\in(-1,-\lambda)$ be the point in $\zeta$-plane
(see Fig.~8) which
corresponds to $D_{n}$ in $z$-plane. Then $g(\zeta,-\kappa_{n_l})
\to g(\zeta,-\kappa)$.
Now (\ref{result}) implies
\begin{equation}\label{newstar}
\lim_{n_l\to\infty}(\theta_{n_l}(\zeta)-n_lg(\zeta)+g(\zeta,-\kappa_{n_l}))=
\frac{1}{2}\ln\frac{\eta-g(\zeta)}{\eta+g(\zeta)}+g(\zeta,\nu).
\end{equation}
The right hand side is independent of
the subsequence $\{ n_l\}$, so the limit as $n\to\infty$ exists
in the left hand side. In the resulting formula we let $\zeta\to\nu$
and obtain
\begin{equation}\label{newbull}
\lim_{n\to\infty}(\theta_n(\nu)-ng(\nu)+g(\nu,-\kappa_n))=
-\frac{1}{2}\ln\frac{2\eta}{\eta_1}+\eta_2.
\end{equation}
The functions $g(\zeta,-\kappa_n)$
are uniformly bounded and
have bounded derivatives on $(\lambda,1)$. Therefore,
$$a_n=\max\{\theta_n(\xi):\lambda<\xi<1\}=\theta_n(\nu)+o(1),\quad
n\to\infty.$$
Thus
\begin{equation}\label{12b}
a_n=\eta n-G(D_n,C)-\frac{1}{2}\ln\frac{2\eta}{\eta_1}+\eta_2+o(1).
\end{equation}
To obtain the final result, this expression for $a_n$ has to be
substituted to (\ref{12a}).

We can simplify the expression $e^{G(D_n,C)}$ in the resulting
formula and avoid solving
equation (\ref{12c}) in the following way.

Let $F$ be the conformal map of the upper half-plane onto a rectangle
$(0,p,p+i,i)$, where $p>0$ and the vertices of the rectangle
correspond to $(1,B,-A,-1)$ in this order.
It is easy to see that $\omega=\Im F$. So
\begin{equation}\label{phiC}
F(C)=i\omega(C),
\end{equation}
and in view of (\ref{12c})
\begin{equation}
\label{phiD}
F(D_n)=p+i\omega(D_n)=p+i\{\alpha n+\omega(C)\}.
\end{equation}
Christoffel--Schwarz formula gives (\ref{harmeasure}), and $p=\tau/i,$
where $\tau$ is defined by (\ref{aaa}). We reflect our
rectangle with respect to the imaginary axis and apply
the map $z\mapsto (i\pi/p)z$, to obtain the
new rectangle
$$R=\{ x+iy:-\pi/p<x<0,\;|y|<\pi\}.$$
Then $e^z$ maps this rectangle $R$ into a ring $e^{-\pi/p}<|w|<1$
and we use the expression of the Green function of
this ring \cite[\S55 (4)]{A3} substituting to this formula\footnote{
In the English edition of 1990, this formula contains two
misprints: an extra vertical line and
missing subscript $1$ in the theta-function in the denominator.}
$\ln w=i\pi-(\pi/p)\omega(D)$, $\ln c=(\pi/p)\omega(C)$
and using $\tau$ instead of $-1/\tau$.
The result is simplified using Table VIII in \cite{A3} and we obtain
$$\displaystyle e^{G(D_n,C)}=
\left|
\frac{\vartheta_0\left(\frac{1}{2}(\{ n\omega(\infty)+\omega(C)\}
-\omega(C))|\,\tau\right)}
{\vartheta_0\left(\frac{1}{2}(\{ n\omega(\infty)+\omega(C)\}
+\omega(C))|\,\tau\right)}\right|,
$$
where $\tau=ip$ is given by (\ref{aaa}).
Combining this with (\ref{12a}) and (\ref{12b}) we obtain the statement
of Theorem~1.1.

\section{Example (the symmetric case)}
We consider the case $I=[-A,-1]\cup[1,A]$. In this case
\begin{equation}\label{ex1}
G(z,\infty)=\int_{A}^z\frac{xdx}{\sqrt{(x^2-1)(x^2-A^2)}}.
\end{equation}
Therefore
\begin{equation}\label{ex2}
\eta=\int_0^1\frac{xdx}{\sqrt{(x^2-1)(x^2-A^2)}}=
\frac 1 2\int_1^{\frac{A^2+1}{A^2-1}}\frac{dt}{\sqrt{t^2-1}}=
\frac 1 2\ln\frac{A+1}{A-1}
\end{equation}
and
\begin{equation}\label{ex3}
\eta_1=-\frac 1  2 G''(0,\infty)=\frac 1 {2A}.
\end{equation}
Also,
\begin{equation}\label{ex4}
G(z,0)=\int_{-1}^z\frac{Adx}{x\sqrt{(x^2-1)(x^2-A^2)}}
\sim\ln \frac 1 z+\ln\frac{2A}{\sqrt{A^2-1}}.
\end{equation}
Notice that $\omega(\infty)=1/2$, $C=0$ and $\omega(C)=1/2$.
Therefore for $n=2m+2$ we have $D_n=\infty$, so $L_{2m+2}=L_{2m+1}$.

For $n=2m+1$ we get
\begin{equation}\label{ex5}
\begin{split}
\sqrt{(2m+1)\eta}\sqrt{\frac{\eta_1}{2\eta}}
e^{(2m+1)\eta+\eta_2}L_{2m+1}\\ \\=
\sqrt{\frac{2m+1}{4A}}\left(\frac{A+1}{A-1}\right)^m\sqrt{\frac{A+1}{A-1}}\frac{2A}{\sqrt{A^2-1}}L_{2m+1}.
\end{split}
\end{equation}
Finally,
\begin{equation}\label{ex51}
\lim_{m\to\infty}
\sqrt{2m+1}\left(\frac{A+1}{A-1}\right)^m
\frac{\sqrt{A}}{A-1}L_{2m+1}=\sqrt{\frac 2{\pi}},
\end{equation}
as we proved in \cite{EY}.

\section{Approximation of $\sgn(x)$ by entire functions
on $[-A,-1]\cup[1,+\infty)$}

Only a minor variation of our method is needed to investigate
the following problem: {\em minimize
\begin{equation}\label{mmm}
\sup\{|f(x)-\sgn(x)|:x\in[-A,-1]\cup[1,+\infty)\}
\end{equation}
among all entire functions $f$ of order $1/2$, type $\sigma$.}

Let $E(\sigma)$ be the infimum (\ref{mmm}). It is easy to prove the
existence of an extremal function using normal families arguments.

Now we describe a construction of extremal functions.
We take the
error $E$ as an independent parameter.
Let $a>0$ be the unique solution of the equation
$L(a)=E$, where
$L(a)$ is defined in the beginning of Section~2.
For $h\geq 0$, and an integer $k\geq 2$, let $\Pi_k(h)=\Pi^-_{k,\infty}(h)$,
that is
the region
in the upper half-plane  whose boundary with respect to the
upper half-plane consists of the segment $[0,ia]$,
and the curve $\delta_{-k}$
as in (\ref{ddd}). Let $\Theta_{k,h}:\bbC_+\to\Pi_k(h)$ be the conformal
map normalized by
$\Theta_{k,h}(\pm1)=0,\Theta_{k,h}(\infty)=\infty$.

\begin{proposition}\label{proa}
If $h=0$ then $S(\Theta_{k,0},a)$ is the unique extremal
function for
\begin{equation}\label{jjj}
A\in[\Theta^{-1}_{k,0}(-c_{k}),\Theta^{-1}_{k,0}(-c_{k-1})].
\end{equation}
If $h>0$ then $S(\Theta_{k,h},a)$ is the unique
extremal function for $$A=\Theta_{k,h}(-c_{k-1}+0).$$
\end{proposition}

The proof of this theorem is similar to the proof of Theorem~3 in \cite{EY}.
We recall the argument for the reader's convenience.

\begin{proof}
Let $f(z)=S(\Theta(z),a)$. Let $x_1<x_2,\ldots\to+\infty$ be the sequence
of all alternance points. Let $\sigma>0$ be the type of $f$
with respect to the order $1/2$. Let $g$ be an entire function of the same
type $\sigma$, order $1/2$. Without loss of generality we may
assume that $g$ is real. Then there exists a sequence $\{ y_k\}$
{\em interlaced} with $\{ x_k\}$, that is,
$$x_1\leq y_1\leq x_2\leq y_2\leq\ldots,$$
such that $f(y_k)=g(y_k)$.
Consider the product
$$F(z)=\prod_{k=1}^\infty\frac{1-z/x_k}{1-z/y_k}.$$
This product converges uniformly on compact subsets of the plane
and has imaginary part of a constant sign in the upper half-plane
and of the opposite sign in the lower half-plane \cite[VII, Thm1]{Lev}.
This implies that
\begin{equation}\label{etype}
F(re^{it})=O(r),\quad r\to\infty
\end{equation} uniformly with respect to $t$ in $\epsilon<t<2\pi-\epsilon$,
for every $\epsilon>0$.
As $f(y_k)=g(y_k)$, we have
\begin{equation}\label{deviat}
\frac{f(z)-g(z)}{f'(z)}=\frac{P(z)}{(z-c)F(z)},
\end{equation}
where $c$ is the critical point of $f$ which is outside the set
$[-A,-1]\cup[1,\infty)$. If there is no such point $c$, then the factor
$(z-c)$ has to be omitted in (\ref{deviat}).
Then $P$ is an entire function of order $1/2$.

Now we notice that the left hand side of (\ref{deviat}) is bounded for
$|\Im z|>1$. Indeed, $g$ and $f-g$ are at most of type $\sigma$, order
$1/2$, while $f'$ has indicator $\sigma\sin(t/2),\; 0<t<2\pi$,
so the ratio has zero type in $\bbC\backslash\bbR_+$ and thus
this ratio is bounded by the Phragm\'en--Lindel\"of theorem.

Combining this with (\ref{etype}) we conclude that $P$ is
a polynomial, and
\begin{equation}\label{bbla}
P(z)/(z-c)=O(z),\; z\to\infty,
\end{equation}
if the
point $c$ exists, and
\begin{equation}\label{bblb}
P(z)=O(z),\; z\to\infty,
\end{equation} if the point
$c$ does not exist.

On the other hand, $P(x)=0$ for each non-critical alternance point $x$.
From our construction of $f=S(\Theta,a)$ it follows that
when $c$ is present, then there are $3$ non-critical alternance points,
namely $-A,-1,1$, while when $c$ is absent, then there are at least
$2$ non-critical alternance points, namely $-1,1$. Together with
(\ref{bbla}), (\ref{bblb}) this implies that $P=0$, that is, $f=g$.
\end{proof}

\begin{proposition}\label{prob}
For every $E$ and $A$ there exist $k,h$ and $a$ such that $S(\Theta_{k,h},a)$
is an extremal function for the set $[-A,-1]\cup[1,+\infty)$.
\end{proposition}

\begin{proof}
For given $E$ we can choose $a$ such that $L(a)=E$.
To prove the existence of $k$ and $h$,
we use a monotonicity argument as in Remarks in Section~3.
Namely, we introduce the following order relation on the pairs
$(k,h)$: $(k,h)\prec(k',h')$ if $k<k'$ or $k=k'$ and $h>h'$.
With this order, the set of pairs $(k,h)$ becomes isomorphic
to the positive ray,
and the correspondence $(k,h)\mapsto A$ becomes
monotone increasing. This function is continuous for $h\neq 0$ and
has a jump at each point $(k,0)$ (this jump is seen in the
right hand side of (\ref{jjj})). So we can obtain any $A>1$
from some pair $(k,h)$.
\end{proof}

\begin{theorem} For every $A$ and $\sigma$, there exists a unique
extremal function $f$ of type $\sigma$, and $f=S(\Theta_{k,h},a)$
for some positive integer $k$, $h\geq 0$ and $a>0$.
\end{theorem}

\begin{proof} Let $\sigma(A,E)$ be the type (with respect to order $1/2$)
of the extremal function defined in Proposition~\ref{prob}. Then
Proposition~\ref{proa} implies
that for every $A$, the function $E\mapsto\sigma(A,E)$ is strictly
decreasing. It is easy to check that
$\sigma(A,1)=0$ and $\sigma(A,0+)=+\infty$.
Moreover, $E\mapsto\sigma(A,E)$ is continuous. So there is
a unique $E=E(A,\sigma)$, which is the error of the best approximation
for given $A$ and $\sigma$, and from this $E$ and $A$ we define $k$ and $h$
using Proposition~\ref{prob}.
\end{proof}
\vspace{.1in}

To state the asymptotic result, we introduce the Martin function
$M(x)$ of the region $\bbC\backslash I$, where
$I=[-A,-1]\cup[1,+\infty)$,
replacing the Green function which we used before.
Martin's function is characterized by the properties that it is
positive and harmonic in $\bbC\backslash I$, equals zero on $I$ and
has asymptotic behavior $$M(-x)\sim\sqrt{x},\; x\to+\infty.$$
We have $M(z)=\Im\cM(z)$ where $\cM$ is the conformal map
of the upper half-plane onto the region
$$\{ x+iy:x>-\pi\alpha,\; y>0\}\backslash[0,i\eta],$$
such that
$$\cM(\pm1)=0,\; \cM(-A)=-\pi\alpha,\;
\cM(-x)\sim \sqrt{x},\; x\to+\infty.$$
These relations define $\alpha$ and $\eta$ uniquely.

Martin's function
has a single critical point $C\in(-1,1)$ and we use the notation
$\eta=M(C)$ and $\eta_1=-M^{\prime\prime}(C)/2,$ as before.
The Green function $G(x,C)$ satisfies
$$G(x,C)=-\ln|x-C|+\eta_2+O(x-C),\quad x\to C,$$
and this defines $\eta_2$.
We also introduce the harmonic measure
$\omega(z)=\omega(z,[-A,-1],\bbC\backslash I)$.
Then $\omega(x)$ is continuous and strictly increasing on $[-\infty,-A)$,
and maps this ray onto $[0,1)$, so
the equation
$$\omega(D_\sigma)=\left\{\alpha\sigma+\omega(C)\right\},$$
where $\{ x\}$ is the fractional part of $x$,
has a unique solution for every $\sigma>0$.

\begin{theorem}
The error of the best uniform approximation of the function
$\sgn(x)$ on $[-A,-1]\cup[1,+\infty)$ by entire functions of
order $1/2$, type $\sigma$ satisfies
$$E(\sigma)\sim \sqrt{\frac{2}{\pi}}(a(\sigma))^{-1/2}e^{-a(\sigma)},$$
where
\begin{equation}\label{llast}
a(\sigma)=\eta\sigma-G(D_\sigma,C)-\frac{1}{2}\ln\frac{2\eta}{\eta_1}+\eta_2.
\end{equation}
\end{theorem}

The equation (\ref{llast}) is analogous to (\ref{12b}).
One can simplify $e^{G(D_\sigma,C)}$ as we did in Section~7
by using an expression
of the Green function in terms of theta-functions.

\bibliographystyle{amsplain}

\bigskip
 Department of mathematics,

 Purdue University,

 West Lafayette, IN 47907, USA

\smallskip
 \textit{E-mail address:}

 eremenko@math.purdue.edu
 \bigskip

Abteilung f\"ur Dynamische Systeme und Approximationstheorie,

 Johannes Kepler Universit\"at Linz,

A-4040 Linz, Austria

\smallskip
\textit{E-mail address:}

Petro.Yudytskiy@jku.at

\end{document}